\documentclass[a4paper,11pt]{article}

\usepackage[utf8]{inputenc}
\usepackage[english]{babel}
\usepackage[T1]{fontenc}

\usepackage[pdftex]{graphicx}
\usepackage{setspace}
\usepackage{indentfirst}

\usepackage{courier} 
\usepackage[font=small,format=plain,labelfont=bf,up,textfont=it,up]{caption}
\usepackage[usenames,svgnames,dvipsnames]{xcolor}
\usepackage[a4paper,top=2.54cm,bottom=2.0cm,left=2.0cm,right=2.54cm]{geometry} 
\usepackage[pdftex,plainpages=false,pdfpagelabels,colorlinks=true,citecolor=black,linkcolor=black,urlcolor=black,filecolor=black,bookmarksopen=true]{hyperref}

\usepackage[round,sort,nonamebreak]{natbib} 

\graphicspath{{./figuras/}} 
\fontsize{60}{62}\usefont{OT1}{cmr}{m}{n}{\selectfont}
\normalsize

\usepackage{amsmath}
\usepackage{amssymb}
\usepackage{amsthm}
\usepackage{mathrsfs}
\numberwithin{equation}{section} 

\newcommand{\AAA}{{\mathcal{A}}}

\newcommand{\CC}{{\mathcal{C}}}
\newcommand{\DD}{{\mathcal{D}}}

\newcommand{\FF}{{\mathcal{F}}}

\newcommand{\GG}{{\mathcal{G}}}

\newcommand{\R}{{\mathbb{R}}}
\newcommand{\E}{{\mathbb{E}}}
\newcommand{\Prob}{{\mathbb{P}}}
\newcommand{\N}{{\mathbb{N}}}

\newcommand{\Nz}{{\mathbb{S}}}
\newcommand{\Nzb}{{\bar{\Nz}}}

\newcommand{\qc}{{\emph{a.s.}} }
\newcommand{\ind}{{\mathbb{I}}}

\newtheorem{teorema}{Theorem}[section]
\newtheorem{lema}[teorema]{Lemma}
\newtheorem{proposicao}[teorema]{Proposition}
\newtheorem{corolario}[teorema]{Corollary}
\newtheorem{observacao}[teorema]{Remark}
\newtheorem{definicao}{Definition}[section]

\title{Elementary results on K processes with weights}
\author{Luiz Renato Fontes \footnote{IME-USP, Rua do Mat\~ao 1010, 05508-090
S\~ao Paulo SP,  Brazil, lrenato@ime.usp.br} \thanks{Partially
supported by CNPq grant 305760/2010-6, and FAPESP grant 2009/52379-8} 
\and Gabriel R.~C.~Peixoto \footnote{IME-USP, Rua do Mat\~ao 1010, 05508-090
S\~ao Paulo SP,  Brazil, gabrielp@ime.usp.br} \thanks{Supported by 
CNPq fellowships 131984/2009-8, 140177/2012-4}}
\date{\today}

\begin{document}

\onehalfspacing



\maketitle

 \begin{abstract} 
 
We introduce the title process via a particular construction, and relate it
to processes previously studied, in particular a process introduced by G.~E.~H.~Reuter
in 1969. We derive elementary properties and quantities of this processes: Markov
property, transition rates, stationary distribution, and the infinitesimal generator 
for a case not treated by Reuter.

 \end{abstract}


\smallskip


\section{Introduction}

The aim of this paper is to introduce and analyse a class of Markov processes 
on a countable state space, which we will call K processes. This terminology has 
been used before in~\cite{fontes:08} for a subclass of the class we consider
here. In comparison, we might call the processes of the present paper 
K processes with weights, and the processes previously studied in that reference,
K processes with uniform weights.

Let us start with a tentative description of our process. Let $\Nz$ 
denote a countably infinite set --- our tentative state space. We want our process
to do the following: when at $x\in\Nz$, it waits an exponential time 
of mean $\gamma_x$, and then jumps to site $y\in\Nz$ with probability
proportional to $\lambda_y$, where $\Lambda:=\{\lambda_x,\,x\in\Nz\}$ 
is a set of weights. This would make immediate sense once we assumed 
that the sum of the weights in $\Lambda$ --- the total weight of $\Nz$ 
--- is finite.

But we make the opposite assumption, namely that the total weight of $\Nz$ 
is infinite. This is natural in the context of scaling limits
of trap models --- more about those models below. Since $\Nz$ is infinite, that 
assumption will then force a condition on the mean waiting times 
$\{\gamma_x,\,x\in\Nz\}$, and we further need to introduce an extra point in the
state space.

A more precise, even if still rough description of our process is as follows. 
Let $\infty$ denote the above mentioned extra point. Our process lives on
$\Nzb=\Nz\cup\{\infty\}$. When at $x\in\Nz$, it waits an exponential time 
of mean $\gamma_x$ and then jumps to $\infty$, which is an instantaneous state. 
Starting at $\infty$, the process enters finite sets $F\in\Nz$ with distribution 
proportional to $\{\lambda_x,\,x\in F\}$. We assume our parameter set 
$\{\gamma_x,\lambda_x,\,x\in F\}$ satisfies $\gamma_x,\lambda_x>0$
for each $x\in\Nz$ and
\begin{equation}
  \label{eq:soma-lambda-gamma}
  \sum_{x \in \Nz} \lambda_x = \infty,\quad
  \sum_{x \in \Nz} \lambda_x \gamma_x < \infty.
\end{equation}
There is a further parameter $c\geq0$ related to the size of the $\infty$-set 
of the process (i.e., the set of times that the process spends visiting $\infty$).
A precise definition/construction of the process will be given in the next section.

The above mentioned connection with trap models is as follows. 
K processes arise as scaling limits of those models under suitable scaling.
This was established in~\cite{fontes:08} for the (symmetric) trap model on the 
complete graph, where the {\em uniform weight} case $\lambda_x\equiv1$ was 
introduced and studied. More recently, the K process with a particular set 
of weights was introduced and shown to be the scaling limit of
{\em asymmetric} trap models on the complete graph in~\cite{bezerra:12}. It also
appears as a scaling limit of the trap model studied in~\cite{jara:12}.
All the K processes showing up in this context have $c=0$. 

The $c>0$ also has a history, albeit a different, older one. The uniform weight
case corresponds to the famous K1 example of Kolmogorov (introduced in~\cite{kolmogorov:51})
of a process in a countable state space with an instantaneous state. 
This example has prompted many studies since its appearance, in particular~\cite{kendall:56}, 
where this example was further analysed, and~\cite{reuter:69},
where the weighted extension was introduced
and analysed. These studies are analytical ones, defined via $Q$-matrices. 
In particular, one common interest is in the derivation of the infinitesimal generators.

What we do in this paper is to introduce the K process via a particular 
construction in Section~\ref{cons}. We establish the Markov property in Section~\ref{markov}, 
with Section~\ref{truncation} dedicated to an auxiliary truncated process. 
Transition rates and stationary distribution are obtained in Section~\ref{trans} 
--- these derivations are absent in~\cite{fontes:08} and~\cite{bezerra:12}). 
It follows from the analysis performed in Section~\ref{trans}
that ours is a construction of the process analytically introduced in~\cite{reuter:69};
as far as we know, such a construction was missing in the general weighted case.
Finally, in Section~\ref{ig}, we compute the infinitesimal generator in the 
case $c=0$; the $c>0$ case, a considerably different one, was done in~\cite{reuter:69}.

We close this introduction with a word about description via Dirichlet forms. 
\cite{fontes:08} take this point of view for the definition and analysis of the 
uniform weight K process (besides also the constructive approach employed in the 
present paper). 
This is also possible in the general weighted case, but we choose not to exploit 
it in the present paper.

\section{Construction}\label{cons}

The state space is a countably infinite set $\Nzb = \Nz \cup \{\infty\}$, in
which $\infty$ is a symbol not in $\Nz$. The parameters of the
K process are a family of strictly positive real numbers
$\{\gamma_x, \lambda_x: x \in \Nz\}$ as well as a constant $c \geq 0$.
We will call $\{\gamma_x: x \in \Nz\}$ {\em waiting time} parameters, and
$\{ \lambda_x: x \in \Nz\}$ the {\em weights} of the process.
We will impose the restrictions \eqref{eq:soma-lambda-gamma} to these parameters.


\begin{observacao}
  \label{obs:restric-gamma-zero}
  Note that \eqref{eq:soma-lambda-gamma}
  implies that $\inf_{x \in \Nz} \gamma_x = 0$.
\end{observacao}

The probability space $(\Omega, \FF, \Prob)$ in which we will
construct the K process needs to admit the following independent
random variables:
\begin{itemize}
\item $\{ N_x: x \in \Nz\}$: a family of independent Poisson processes
  on $\R^+$, where $N_x$ have rate $\lambda_x$. We will denote the
  marks of $N_x$ by $0 < \sigma_1^x < \sigma_2^x < \ldots$.
\item $\{T_0\} \cup \{T_n^x: n \in \N, x \in \Nz\}$: a family of
  \emph{i.i.d.} random variables, exponential of rate $1$.
\end{itemize}

\begin{definicao}
  \label{def:Gamma}
  Let the {\em Clock Process} be given by:
  \begin{gather}
    \label{eq:Gamma}
    \Gamma(t) := \Gamma^y (t) := 
    \gamma_y T_0 + \sum_{z \in \Nz} \sum_{i = 1}^{N_z(t)} \gamma_z T^z_i
    + c t,
  \end{gather}
  for $t \geq 0$ and $y \in \Nzb$. We adopt the convention that
  $\gamma_\infty = 0$ and $\sum_{i=1}^{0}  \gamma_z T^z_i = 0$.
\end{definicao}

\begin{observacao}
  \label{obs:Gamma-cadlag}
  $\{\Gamma(t): t \geq 0\}$ is \qc strictly increasing and càdlàg. It
  also has stationary and independent increments. For each $t > 0$,
  $\Gamma(t)$ has finite mean and its distribution is absolutely
  continuous with respect to the Lebesgue measure.
\end{observacao}


\begin{definicao}
  \label{def:processo}
  The K process with initial state $y \in \Nzb$ is a stochastic
  process $\{X^y(t): t \geq 0\}$ defined as follows:
  \begin{align}
    \label{eq:def-processo}
    X(t) := X^y(t) := 
    \begin{cases}
      y, & \text{if } t < \gamma_y T_0;\\
      x, & \text{if } t \in \bigcup_{i=1}^{\infty} [
      \Gamma^y(\sigma_i^x-), \Gamma^y(\sigma_i^x)); \\
      \infty, & \text{otherwise.}
    \end{cases}
  \end{align}
\end{definicao}

\begin{observacao}
  \label{obs:jump}
  This definition also makes sense when the first sum in~(\ref{eq:soma-lambda-gamma}) above is finite, 
but in this case the resulting process is a Markovian pure jump process. 
This follows from the fact that in this case the set of Poissonian marks 
$\{\sigma_i^x:\,i\geq1,x\in\Nz\}$ is discrete (i.e., has no limit points) almost surely.
\end{observacao}

\begin{observacao}
  \label{obs:reinicia-infinito}
  Note that, by construction, the process started at a
  site $y \in \Nz$ stays there till the time $\gamma_y T_0$ and
  afterwards behaves as a copy of the K process started at
  $\infty$. That is:
  \begin{displaymath}
      X^{y}(t+\gamma_y T_0) = X^{\infty}(t),
  \end{displaymath}
  almost surely for every $t \geq 0$.
\end{observacao}

\begin{observacao}
  \label{obs:desc}
  It is a straightforward exercise to check that $X$ satisfies the rough description given in the 
  introduction (see beginning of paragraph of \eqref{eq:soma-lambda-gamma}), and that $c$ can be obtained as 
  $\lambda_x$ times the expected length of the $\infty$-set of $X$ between two consecutive 
  visits to $x$, for any $x\in\Nz$.
\end{observacao}

We will equip $\Nzb$ with the following metric:
\begin{equation}
  \label{eq:metrica}
  d(x, y) = (\gamma_x + \gamma_y) \ind\{ x \neq y\}.
\end{equation}
Under this metric $\{x\}$ is a open set for each $x \in \Nz$, and, roughly speaking, 
$x$ is close to $\infty$ when $\gamma_x$ is small. In symbols:
$x\to\infty$ if and only if $\gamma_x\to0$.

\begin{observacao}
  \label{obs:metrica}
  One readily checks that $(\Nzb, d)$ is a Polish space, not necessarily compact.  
  The continuous functions in the topology generated by this
  metric are the set:
  \begin{equation}
    \label{eq:classe-C}
    \CC = \left\{
      f: \Nzb \to \R : 
      \forall \epsilon > 0 \exists \delta > 0 \textrm{ s.~t. }
      |f(x) - f(\infty)| < \epsilon \text{ whenever } \gamma_x < \delta
    \right\}.
  \end{equation}
  Continuous functions in this topology are uniformly continuous.
\end{observacao}

\section{Truncated processes}\label{truncation}

We will establish the Markov property for the K process in the next section.
Our strategy 
is to approximate it by Markov pure jump processes. In this section, we 
will construct these approximating processes in the same probability space 
as the original K process, and then prove an almost sure convergence in 
Theorem~\ref{teo:convergencia-truncados} below.

\begin{definicao}
  \label{def:conj-truncados}
  For each $\delta > 0$, let 
  \begin{align*}
    \Nz_\delta := \{x \in \Nz: \gamma_x \geq \delta\},\quad
    \notag
    \Nzb_\delta := \Nz_\delta \cup \{\infty\}.
  \end{align*}
\end{definicao}

\begin{observacao}
  \label{obs:trunc-lambda}
  Note that the second condition in~\eqref{eq:soma-lambda-gamma}
  implies that $\sum_{x \in \Nz_\delta} \lambda_x < \infty$ for every
  $\delta > 0$.
\end{observacao}

\begin{proposicao}
    \label{prop:processo-cadlag}
    The K process is almost surely càdlàg.
\end{proposicao}

\begin{proof}
  Using Remark \ref{obs:reinicia-infinito}, we need only show that
  the K process started at $\infty$ is càdlàg.

  Following \cite{billingsley:99}, 
  fixing a realization of the process, $T > 0$ and
  $\epsilon > 0$, let us show that there exists $0 = t_0 < t_1 < \ldots <
  t_N = T$ such that $w[t_{i-1}, t_i) < \epsilon$ for every $i = 1,
  \ldots, N$, where $w(A) = \sup_{t, s \in A} d(X^\infty(t),
  X^\infty(s))$.

  Let us fix $\delta < \epsilon / 2$, so the diameter of $\Nzb
  \setminus \Nz_\delta$ is at most $\epsilon$. And take $S_1 < \ldots
  < S_M$ an ordering of the set $\{ \sigma^x_i: x \in \Nz_\delta , \:
  i \geq 1, \: \Gamma(\sigma^x_i) \leq T\}$. This ordering is possible
  by Remark \ref{obs:trunc-lambda}, which implies that the latter set is almost surely finite
  for each finite $T$.

  Finally let us take $N = 2M+1$, $t_0 = 0$, $t_N = T$ and for
  $i=1,\ldots, M$:
    $t_{2i-1} = \Gamma(S_i-)$, 
    $t_{2i} = \Gamma(S_i)$. 
  %
  If $t \in [t_{2i-1},t_{2i})$, then $X(t)$ is constant in
  this interval, while if $t \in [t_{2i-2}, t_{2i-1})$,
  then $X(t) \in \Nzb \setminus \Nz_\delta$, so the variation in this
  interval is at most $\epsilon$.  The same occurs in the intervals
  $[t_0, t_1)$ and $[t_{N-1}, t_N)$. The argument is complete.
\end{proof}

\begin{observacao}
  \label{obs:ponto-fixo-continuo}
  By the argument in the proof of Proposition \ref{prop:processo-cadlag},
  the only candidates for discontinuity points of the K process are
  $\{\Gamma(\sigma_i^x), \Gamma(\sigma_i^x-): i \in \N \, , x \in \Nz
  \}$. Since this a countable set and each such point is an absolutely continuous 
  random variable, then every fixed deterministic $t > 0$ is almost surely a 
  continuity point of the K process.
\end{observacao}

\begin{definicao}
  \label{def:convergencia-truncado}
  The truncated process at level $\delta > 0$ with initial state $y
  \in \Nzb_\delta$ is defined as follows:
  \begin{align}
    \label{eq:processo-truncado}
    X_\delta(t) := X_\delta^{y}(t) := 
    \begin{cases}
      y & \text{if } t < \gamma_y T_0\\
      x & \text{if } t \in \bigcup_{i=1}^{\infty} [
      \Gamma^y_\delta(\sigma_i^x-), \Gamma_\delta^{y}(\sigma_i^x)) \\
      \infty & \text{otherwise;}
    \end{cases}\\
    \Gamma_\delta(t) := \Gamma_\delta^{y} (t) := 
    \gamma_y T_0 + \sum_{z \in \Nz_\delta} \sum_{i = 1}^{N_z(t)} \gamma_z T^z_i
    + c t.
  \end{align}
\end{definicao}

\begin{observacao}
  \label{obs:restric-truncados}
  Since $\sum_{x \in \Nz_\delta} \lambda_x < \infty$ (Remark
  \ref{obs:trunc-lambda}) then the truncated process is a Markov pure
  jump process described as follows:
  \begin{itemize}
  \item In the case $c > 0$, from a state $x \neq \infty$, the process
    stays an exponential time of mean $\gamma_x$, after which it jumps
    to $\infty$. There it stays a exponential time of rate
    $\frac{1}{c} \sum_{x \in \Nz_\delta} \lambda_x$ and chooses a new
    state $y\in \Nz_\delta$ to jump to with probability proportional
    to $\lambda_y$.
  \item The case $c = 0$ is similar, except that in this case the
    process never visits $\infty$. After exiting a state $x \in
    \Nz_\delta$ it chooses a new state $y \in \Nz_\delta$ with
    probability proportional to $\lambda_y$. Note that in this case
    the truncated process never visits $\infty$; it is even true that
    $X^\infty(0) \neq \infty$ \qc.
  \end{itemize}
  So we can conclude that $X_\delta$ is Markovian and càdlàg.  
\end{observacao}

\begin{teorema}
  \label{teo:convergencia-truncados}
  $X_\delta^{\infty} (\bullet)$ converges to $X^\infty(\bullet)$
  \qc in the Skorohod $J_1$ topology as $\delta \to 0$.
\end{teorema}

\begin{proof} 
  Following Proposition 5.3 from Chapter 3 of \cite{ethier:86}, we
  will prove that, for each $T > 0$, exists increasing bijective and
  Lipschitz continuous functions $\lambda_\delta: [0,\infty) \to [0,
  \infty)$, $\delta>0$,  such that almost surely:
  \begin{gather*}
    \lim_{\delta \to 0} \sup_{0 \leq t \leq T} d\left( X_\delta^{\infty}(t),
      X^\infty(\lambda_\delta(t)) \right) = 
    \lim_{\delta \to 0} \sup_{0 \leq t \leq T} |\lambda_\delta(t) - t| = 0.
  \end{gather*}
  
  These functions may depend on the realization of the process.  As the
  process always starts at $\infty$, we will stop carrying this index
  in our notation.
  
  Fix a realization of the process and a $T > 0$ and take $0 =
  S_\delta^0 < S_\delta^1 < \ldots $ an ordering of $\{0\} \cup \{
  \sigma^x_i : x \in \Nz_\delta, \, i \geq 1\}$.  This is a.s.~possible 
  since $\sum_{x\in \Nz_\delta} \lambda_x < \infty$.
  
  For $0<\epsilon < \delta$ let us define:
    \begin{displaymath}
    L_\delta^\epsilon := \min \left\{ i \geq 1: \Gamma_\epsilon(S_\delta^i) \geq T \right\}.
  \end{displaymath}
  
  The set $\{\sigma_i^x: x \in \Nz\setminus\Nz_\delta, i\geq 1\}$ is
  dense, so for $\epsilon$ sufficiently small we have that
  $\Gamma_\epsilon(S_\delta^{i+1}-) >
  \Gamma_\epsilon(S_\delta^i)$ for every $i < L^\epsilon_\delta$.
  
  For such $\epsilon$, we define $\lambda^\epsilon_\delta: [0, \infty)
  \to [0, \infty)$ the following way:
  \begin{displaymath}
    \lambda^\epsilon_\delta(t) = \begin{cases}
      \Gamma(S^i_\delta) + \frac{\Gamma(S^{i+1}_\delta-) - \Gamma(S^i_\delta)}
      {\Gamma_\epsilon(S^{i+1}_{\delta} -) - \Gamma_\epsilon(S^i_\delta)}
      \left[t - \Gamma_\epsilon(S^i_\delta)\right]
      & \textrm{ if }
      \Gamma_\epsilon(S^i_\delta) \leq t \leq \Gamma_\epsilon(S^{i+1}_\delta-) \\
      \Gamma(S^{i+1}_\delta-) - \Gamma_\epsilon(S^{i+1}_\delta-) + t
      & \textrm{ if }
      \Gamma_\epsilon(S^{i+1}_\delta-) \leq t \leq \Gamma_\epsilon(S^{i+1}_\delta)\\
      \Gamma(S_\delta^{L^\epsilon_\delta}) - \Gamma_\epsilon(S_\delta^{L^\epsilon_\delta}) + t
      & \textrm{ if } t > \Gamma_\epsilon(S^\delta_{L^\epsilon_\delta})
    \end{cases}
  \end{displaymath}

  Note that for $i =  0, \ldots L_\delta^\epsilon$:
  \begin{align*}
     \lambda^\epsilon_\delta(\Gamma_\epsilon(S^i_\delta-)) = \Gamma(S^i_\delta-),\quad
     \lambda^\epsilon_\delta(\Gamma_\epsilon(S^i_\delta)) = \Gamma(S^i_\delta),
  \end{align*}
  while $ \lambda^\epsilon_\delta$ is a linear interpolation between
  these points.
  
  Since $\Gamma(t) \geq \Gamma_\epsilon(t)$ for every $\epsilon > 0$,
  we have that $ \lambda^\epsilon_\delta(t) \geq t$. Using the
  linearity by parts and the fact that $\Gamma(\sigma^x_i) =
  \Gamma(\sigma_i^x-) + \gamma_x T^x_i$ and
  $\Gamma_\epsilon(S_\delta^{L_\delta^\epsilon}) \geq T$, we find
  that:
  \begin{displaymath}
    \sup_{0 \leq t \leq T} | \lambda^\epsilon_\delta(t) - t| \leq
    \max_{0 \leq i \leq L^\epsilon_\delta} \{ \Gamma(S^i_\delta) -
    \Gamma_\epsilon(S^i_\delta) \}.
  \end{displaymath}

  This quantity vanishes almost surely if $\delta$ is fixed and
  $\epsilon \to 0$.  So, for every $\delta > 0$ there exists 
  $\epsilon(\delta) > 0$ such that:
  \begin{displaymath}
    \sup_{0 \leq t \leq T} | \lambda^\epsilon_\delta(t) - t| < \delta
  \end{displaymath}
  for every $\epsilon \leq \epsilon(\delta)$.

  We can take $\epsilon(\delta)$ so that it is decreasing in
  $\delta$. So we can take $\delta(\epsilon)$ as its inverse. Since
  $\epsilon(\delta) > 0 \forall \delta$, then $\lim_{\epsilon \to 0}
  \delta(\epsilon) = 0$.  Thus:
  \begin{displaymath}
    \sup_{0 \leq t \leq T} |\lambda^\epsilon_{\delta(\epsilon)}(t)-t| <
    \delta(\epsilon) \xrightarrow{\epsilon \to 0} 0.
  \end{displaymath}
  
  By construction we have that $t \in
  [\Gamma_\epsilon(S^{i}_\delta-),
  \Gamma_\epsilon(S^{i}_\delta))$ if and only if
  $ \lambda^\epsilon_\delta(t) \in [\Gamma(S^{i}_\delta-),
  \Gamma(S^{i}_\delta))$. So, if $x \in \Nz_\delta$, we have that
  $X( \lambda^\epsilon_\delta(t)) = x$ if and only if
  $X_\epsilon(t) = x$. Observing that the diameter of $\Nzb
  \setminus \Nz_\delta$ is smaller or equal than $2 \delta$, we
  conclude that:
  \begin{displaymath}
    \sup_{0 \leq t \leq T} d\left(X( \lambda^\epsilon_\delta(t)),
      X_\epsilon (t)\right) 
    \leq 2 \delta.
  \end{displaymath}

  Making ${\lambda}_\epsilon =
  \lambda^\epsilon_{\delta(\epsilon)}$, we conclude that:
  \begin{displaymath}
    \sup_{0 \leq t \leq T} |{\lambda}_\epsilon(t) - t|
    \xrightarrow{\epsilon \to 0} 0, \quad
    \sup_{0 \leq t \leq T} d(X({\lambda}_\epsilon(t)), X_\epsilon(t))
    \xrightarrow{\epsilon \to 0} 0
    \qedhere
  \end{displaymath}
\end{proof}

\begin{corolario}
  \label{cor:convergencia-truncados}
  For every $y \in \Nzb$, $T > 0$ and almost every realization of the
  K process, there are increasing bijective and Lipschitz continuous
  functions $\lambda^{(y)}_\delta: [0, \infty) \to [0, \infty)$ such that:
  \begin{align}
    \lambda^{(y)}_\delta(t) &\geq t, \notag\\
    \label{eq:convergencia-truncado-uniforme}
    \sup_{0 \leq t \leq T} |\lambda^{(y)}_\delta(t) - t| &\leq
    \sup_{0 \leq t \leq T} |\lambda^{(\infty)}_\delta(t) - t|
    \xrightarrow{\delta \to 0} 0,\\
    \sup_{0 \leq t \leq T} d(X^y(\lambda_\delta^{(y)}(t)), X_\delta^{y}(t)) &\leq
    \sup_{0 \leq t \leq T} d(X^\infty(\lambda_\delta^{(y)}(t)), X_\delta^{\infty}(t))
    \xrightarrow{\delta \to 0} 0. \notag
  \end{align}

  Thus $X_\delta^{y}$ converges \qc to $X^y$ in the Skorohod
  $J_1$ topology as $\delta \to 0$ uniformly in $y$.
\end{corolario}
\begin{proof}
  After fixing a realization and a $T > 0$, let us take
  $\lambda_\delta: [0, \infty) \to [0, \infty)$ as in the proof of
  Theorem \ref{teo:convergencia-truncados}.
  
  Take  $\lambda^{(\infty)}_\delta = \lambda_\delta$ and for $y \in \Nz$ define:
  \begin{displaymath}
    \lambda_\delta^{(y)}(t) = \begin{cases}
      t, & \textrm{ if } t < \gamma_y T_0\\
      \gamma_yT_0 + \lambda_\delta(t - \gamma_y T_0),
      & \textrm{ if } t \geq \gamma_y T_0
    \end{cases}
  \end{displaymath}

  Observing Remark \ref{obs:reinicia-infinito} we conclude that this
  family of functions satisfies
  \eqref{eq:convergencia-truncado-uniforme}.
\end{proof}

\section{Markov property}\label{markov}

\begin{definicao}
  \label{def:semigrupo}
  We will denote the semigroups of the K process and the truncated
  process by $(\Psi_t)_{t \geq 0}$ and $(\Psi^{\delta}_t)_{t \geq
    0}$. That is, for $f: \Nzb \to \R$ we define:
  \begin{align*}
    \Psi_t f (x) := \E \left[  f(X^x(t) \right],\quad
    \Psi_t^{\delta} f (x) := \E \left[  f(X_\delta^{x}(t) \right] \\
  \end{align*}
\end{definicao}

\begin{proposicao}
  \label{prop:quase-feller}
  Let $f: \Nzb \to \R$ be a bounded continuous function, then $\Psi_t
  f$ is also a bounded continuous function for every $t > 0$.
\end{proposicao}

\begin{observacao}
  This result establishes a property which is close to (but is not quite) 
  a Feller property of our semigroup. For a semigroup to be
  Feller, it would have to take continuous functions that vanish at
  infinity into continuous functions that vanish at infinity. This may
  not be the case for our process, depending on the choice of the
  parameters.
\end{observacao}

\begin{proof}[Proof of Proposition \ref{prop:quase-feller}]

  Fix a $t > 0$. If $f$ is bounded, then obviously $\Psi_t f$ is also
  a bounded function. We are left with showing that it is continuous.  For
  that take $(x_n)$ a sequence of elements of $\Nzb$ that converges to
  $x \in \Nzb$. We want to show that:
  \begin{equation}
    \label{eq:quase-feller-1}
    \lim_{n \to \infty} \Psi_t f (x_n) = \Psi_t f (x)
  \end{equation}

  If $x \neq \infty$, then observing the metric \eqref{eq:metrica} we
  notice that $d(x_n, x) \geq \gamma_x\ind\{ x_n = x \}$, so there
  must exist an $n_0$ such that $x_n = x$ for every $n > n_0$. 
  \eqref{eq:quase-feller-1} follows immediately in this case.

  If $x = \infty$, we can write:
  \begin{displaymath}
    \left| \Psi_t f (x_n) - \Psi_t f (x) \right| =
    \left| \E \left[ 
      f (X^{x_n}(t)) - f (X^{\infty}(t))
    \right] \right| \leq
    \E \left| 
      f (X^{x_n}(t)) - f (X^{\infty}(t))
    \right|.
  \end{displaymath}

  The quantity inside the expected value is bounded (since $f$ is
  bounded), so if we can show that this quantity converges almost
  surely to zero then we can conclude with the dominated convergence
  theorem.

  Note that $d(x_n, \infty) = \gamma_{x_n} \xrightarrow{n\to\infty}
  0$, so for a fixed realization of the process we can take $n$ large
  enough so $\gamma_{x_n} T_0 < t$. In this case, using Remark
  \ref{obs:reinicia-infinito}, we can write:
  \begin{displaymath}
    f (X^{x_n}(t)) - f (X^{\infty}(t)) = 
    f (X^\infty(t-\gamma_{x_n} T_0)) - f (X^\infty(t)).
  \end{displaymath}

  This quantity vanishes as $n \to \infty$ since $t$ is \qc a
  continuity point of the K process (by Remark
  \ref{obs:ponto-fixo-continuo}), $\gamma_{x_n} \to 0$ as $n \to
  \infty$ and $f$ is continuous.
\end{proof}

\begin{observacao}
  One way to turn the K process into a Feller process if it does not
  have that property in the above formulation, as pointed out to us by
  Milton Jara, is to add new sites to $\Nz$ other than only $\infty$
  as follows. For concreteness, let us suppose in this remark that
  $\Nz = \{1,2, \ldots \}$.  Let $\GG$ be the set of limit points of
  $\{\gamma_x: x \in \Nz\}$. Note that $0 \in \GG$ by Remark
  \ref{obs:restric-gamma-zero}.  The new state space would be:
  \begin{displaymath}
    \Nzb := \{ (x, \gamma_x): x \in \Nz\} \cup \{ (\infty, g): g \in
    \mathcal{G} \},
  \end{displaymath}
  equipped with the metric:
  \begin{equation}
    \label{eq:metrica-concertar-feller}
    d( (x, \gamma_x), (y, \gamma_y) ) = \left| \frac{1}{x} - \frac{1}{y}
    \right| + |\gamma_x - \gamma_y|,
  \end{equation}
  with the convention that $\frac{1}{\infty} = 0$. This can be seen as a 
  completion of the metric \eqref{eq:metrica-concertar-feller} applied 
  to the original space.

  For a point $(y, g) \in \Nzb$, we (re)define our process as follows:
  \begin{align*}
    \Gamma^{y, g} (t) &:= g T_0 + \sum_{z \in \Nz} \sum_{i =
      1}^{N_z(t)} \gamma_z T^z_i + c t;\\
    X^{y, g} &:= 
    \begin{cases}
      (y,g), & \text{if } t < g T_0;\\
      (x, \gamma_x), & \text{if } t \in \bigcup_{i=1}^{\infty} [
      \Gamma^{y,g}(\sigma_i^x-), \Gamma^{y,g}(\sigma_i^x)); \\
      (\infty, 0), & \text{otherwise.}
    \end{cases}
  \end{align*}

  The point $(\infty, 0)$ replaces $\infty$ in the original formulation, 
  and the new points $(\infty, g) \in \Nzb$, $g > 0$, are states that are 
  never visited by the K process, except when the process starts in one 
  of them; in this case, the process never visits any of them again after 
  leaving that initial one of them for the first time. 

  All results from this paper, suitably reformulated when necessary,
  hold also for this version of the K process. Additionally, under
  this reformulation, the process satisfies the Feller property.

  A mild complication of this approach is that $\Nzb$ could be
  uncountable, making the characterization of the K process via a $Q$
  matrix, as will be done in Section~ \ref{trans}, cumbersome. That
  and the heavier notation of the new formulation led us to stick to
  the original definition of our process.
\end{observacao}

\begin{lema}
  \label{lema:visitados-compacto}
  For $t > 0$, let $V_t := \cup_{s \in [0, t]} \{ X(s) \}$ be the set
  of the visited states up to time $t$. Then:
  \begin{equation}
    \label{eq:visitados-compacto}
    \#\{ x \in V_t: \gamma_x > \epsilon\} < \infty
  \end{equation}
  almost surely for every $\epsilon > 0$ and $t > 0$.
\end{lema}

\begin{proof}
  Since the events considered are monotonic in $\epsilon$ and $t$,
  then showing that \eqref{eq:visitados-compacto} occurs with
  probability $1$ for any fixed $\epsilon > 0$ and $t > 0$ is enough
  to prove the lemma.
  
  For fixed $\epsilon > 0$ and $t > 0$, consider the complementary
  event, that is:
  \begin{equation}
    \label{eq:visitados-compacto-1}
    \{\#\{ x \in V_t: \gamma_x > \epsilon\} = \infty\}
  \end{equation}
  
  In this event, for any $N \in \N$ fixed, there are states $x_1,
  \ldots x_N$ such that:
  \begin{align*}
    \epsilon \sum_{i=1}^N T^{x_i}_1 <
    \sum_{i=1}^N \gamma_{x_i} T^{x_i}_1 < t .
  \end{align*} 
  
  So we can conclude that, if $S_N$ is a sum of $N$ independent
  exponential random variables of rate $1$, then the probability of
  \eqref{eq:visitados-compacto-1} can be bounded by the probability of
  $\epsilon S_N$ being smaller than $t$.  The lemma follows from the
  observation that $S_N \to \infty$ as $N \to \infty$ with probability
  $1$.
\end{proof}

\begin{lema}
  \label{lema:markov1}
  Take $f: \Nzb \to \R$ a bounded continuous function, $t > 0$ and
  a set $A \subseteq \Nzb$ satisfying:
  \begin{equation}
    \label{eq:convergencia-semigrupo-hip}
    \# \{ x \in A: \gamma_x > \epsilon \} < \infty,
   \end{equation}
  for every $\epsilon > 0$. Under these conditions:
  \begin{equation}
    \label{eq:convergencia-semigrupo}
    \sup_{y \in A} | \Psi^\delta_t f (y) - \Psi_t f(y) |
    \xrightarrow{\delta \to 0} 0.
  \end{equation}
\end{lema}

\begin{proof}

  Let us first show the convergence without uniformity, that is, we
  will show that $\Psi^\delta_t f (y) \to \Psi_t f(y)$ as $\delta \to
  0$ for any $y \in A$. Let us write
  \begin{align}
    \label{eq:convergencia-semigrupo-1}
    \Psi^\delta_t f (y) -\Psi_t f(y) &=
    \E \left[
      f(X_\delta^{y}(t)) - f(X^y(t))
    \right].
  \end{align}
  Remark \ref{obs:ponto-fixo-continuo} guarantees that $t$ is
  almost surely a continuity point of $X^y$. Since convergence in the
  Skorohod topology implies convergence at the continuity points of the
  limit trajectory and $f$ is a continuous function, then 
  $f(X_\delta^{y}(t)) \to f(X^y(t))$ \qc as $\delta \to 0$.

  Since $f$ is bounded, the dominated convergence theorem implies that
  \eqref{eq:convergencia-semigrupo-1} vanishes as well as
  $\delta\to 0$.

  This is readily extended to the case where $A$ is finite. Let us
  suppose from now on that $A$ is infinite.

  Take $\lambda_\delta^{(y)}$ as in Corollary
  \ref{cor:convergencia-truncados} for some $T > t$. We can write:
  \begin{align}
    \notag
    \sup_{y \in A} | \Psi^\delta_t f (y) - \Psi_t f(y) |
    &= \sup_{y \in A} \left\lvert \E \left[f(X_\delta^{y}(t)) -
        f(X^y(t)) \right]\right\rvert \\
    \label{eq:convergencia-semigrupo-2}
    &\leq \sup_{y\in A}\left\lvert \E \left[f(X_\delta^{y}(t)) -
        f(X^y(\lambda_\delta^{(y)}(t))) \right]
    \right\rvert\\
    \label{eq:convergencia-semigrupo-3}
    &+ \sup_{y\in A}\left\lvert \E \left[f(X^y(\lambda_\delta^{(y)}(t))) -
        f(X^y(t)) \right] \right\rvert.
  \end{align}

  Let us now show that both terms of this sum vanish as $\delta
  \to 0$.

  Observing \eqref{eq:convergencia-truncado-uniforme}, then for every
  $y$ \qc:
  \begin{equation}
    \label{eq:convergencia-semigrupo-4}
    d(X_\delta^{y}(t), X^y(\lambda_\delta^{(y)}(t)) \leq
    \sup_{s \in [0, T]} d(X_\delta^{\infty}(s), X^\infty(\lambda_\delta^{(\infty)}(s))
    \xrightarrow{\delta \to 0} 0.
  \end{equation}

  Since $f$ is a continuous functions, it is also uniformly continuous
  (see Remark \ref{obs:metrica}). That is, for every $\epsilon > 0$,
  exists a $\eta_\epsilon > 0$ such that $|f(x)-f(y)| < \epsilon$
  whenever $d(x, y) < \eta_\epsilon$.

  If we fix an arbitrary $\epsilon > 0$, we can bound
  \eqref{eq:convergencia-semigrupo-2} by:
  \begin{align*}
    &\sup_{y\in A} \left\lvert \E \left[ f(X_\delta^{y}(t)) -
        f(X^{y}(\lambda_\delta^{(y)}(t))) \right]
    \right\rvert\\
    \leq&\, \sup_{y\in A} \left\lvert \E \left[ \left(f(X_\delta^{y}(t)) -
          f(X^{y}(\lambda_\delta^{(y)}(t))) \right) \ind\{ d(X_\delta^{y}(t),
        X^y(\lambda_\delta^{(y)}(t)) < \eta_\epsilon \} \right]
    \right\rvert \\
    +&\, 2\Vert f \Vert \sup_{y \in A} \Prob \left( d(X_\delta^{y}(t),
      X^y(\lambda_\delta^{(y)}(t)) \geq
      \eta_\epsilon \right) \\
    \leq&\, \epsilon + 2 \Vert f \Vert \,\Prob \left( \sup_{s \in [0,
        T]}d(X_\delta^{\infty}(s), X^\infty(\lambda_\delta^{(\infty)}(s)) \geq
      \eta_\epsilon \right).
  \end{align*}
  Observing \eqref{eq:convergencia-semigrupo-4}, we conclude that this
  last quantity converges to $\epsilon$ as $\delta \to 0$. Since
  $\epsilon$ is an arbitrary positive constant we conclude that
  \eqref{eq:convergencia-semigrupo-2} vanishes as $\delta \to 0$.
 
  To compute \eqref{eq:convergencia-semigrupo-3}, note that each
  term of that quantity vanishes as $\delta \to 0$ since $t$ is
  \qc a continuity point of $X^y$ and $\lambda_\delta^{(y)}(t) \to t$
  \qc. Using \eqref{eq:convergencia-semigrupo-hip} we can take
  $\eta(\delta)$ a monotonic function that vanishes as $\delta
  \to 0$ and if we define $A_\delta := \{x \in A: \gamma_x <
  \eta(\delta)\}$ then:
  \begin{displaymath}
      \max_{y  \in A \setminus A_\delta} \left\lvert
      \E \left[f(X^y(\lambda_\delta^{(y)}(t))) - f(X^y(t)) \right]
    \right\rvert \xrightarrow{\delta \to 0} 0.
  \end{displaymath}

  Since $\eta(\delta) \to 0$, we may choose $\delta$ so small that
  $\sqrt{\eta(\delta)} < t \leq \lambda_\delta^{(y)}(t)$. The latter
  inequality follows from \eqref{eq:convergencia-truncado-uniforme}.

  Using Remark \ref{obs:reinicia-infinito} and considering in the
  cases $\gamma_yT_0 < \sqrt{\eta(\delta)}$ and $\gamma_yT_0 \geq
  \sqrt{\eta(\delta)}$, we can write:
  \begin{align}
    \label{eq:convergencia-semigrupo-5}
    &\sup_{y \in A_\delta} \left\lvert \E \left[f(X^y(\lambda_\delta^{(y)}(t))) -
        f(X^y(t)) \right]
    \right\rvert \notag \\
    \leq& \sup_{y \in A_\delta} \E \left\lvert
      \left[f(X^\infty(\lambda_\delta^{(y)}(t)-\gamma_yT_0)) -
        f(X^\infty(t-\gamma_yT_0)) \right] \ind\{ \gamma_yT_0 <
      \sqrt{\eta(\delta)} \}
    \right\rvert\\
    +& 2 \Vert f \Vert \sup_{y\in A_\delta} P(\gamma_y T_0 > \sqrt{\eta(\delta)}).
    \notag
  \end{align}

  Note that $\sup_{y\in A_\delta} \gamma_y \leq \eta(\delta)$,
  so $\sup_{y\in A_\delta} \Prob(\gamma_y T_0 > \sqrt{\eta(\delta)})
  \leq e^{-\frac{1}{\sqrt{\eta(\delta)}}} \to 0$ as $\delta \to 0$.

  We can bound the first term in \eqref{eq:convergencia-semigrupo-5}
  by:
  \begin{displaymath}
     \E \left[ \sup_{y \in A_\delta}
      \sup_{0 \leq s \leq \sqrt{\eta(\delta)}} \left\lvert
        f(X^\infty(\lambda_\delta^{(y)}(t)-s)) -
        f(X^\infty(t-s))
    \right\rvert \right].
  \end{displaymath}

  The random variable inside the expected value is bounded by $2\Vert f
  \Vert$.  So if we show that it goes almost surely to zero, then the
  dominated convergence theorem will tell us that its expected value
  also vanishes.

  Fixing an $\epsilon > 0$, take $\epsilon^\prime > 0$ such that
  $|f(x) - f(y)| < \epsilon$ whenever $d(x, y) < \epsilon^\prime$.
  This is possible by the uniform continuity of $f$.
 
  As $t$ is \qc a continuity point of $X^\infty$, there exists an
  $\epsilon^{\prime\prime}$ such that $d(X^\infty(t), X^\infty(t+s))
  <\frac{\epsilon^{\prime}}{2}$ whenever $|s| <
  \epsilon^{\prime\prime}$.

  If we take $\delta_0$ such that $\delta < \delta_0$ implies that
  $\sqrt{\eta(\delta)} < \frac{\epsilon^{\prime\prime}}{2}$ and
  $\sup_{0 \leq s \leq T} |\lambda_\delta^{(\infty)}(s) - s| <
  \frac{\epsilon^{\prime\prime}}{2}$, then for every $y\in\Nzb$ and $s
  \leq \sqrt{\eta(\delta)}$:
  \begin{gather*}
    |\lambda_\delta^{(y)}(t) - s - t| \leq
     |\lambda_\delta^{(y)}(t) - t | + |s| \leq
     \sup_{r \in [0, T]} |\lambda_\delta^{(\infty)}(r) - r | + |s|
     < \epsilon^{\prime\prime},\\
    |t - s - t | = |s| < \epsilon^{\prime\prime}.
  \end{gather*}
  
  So it is valid that:
  \begin{displaymath}
    d(X^\infty(\lambda_\delta^{(y)}(t)-s), X^\infty(t-s)) \leq
     d(X^\infty(\lambda_\delta^{(y)}(t)-s), X^\infty(t))+
     d(X^\infty(t), X^\infty(t-s))
     < \epsilon^{\prime}.
  \end{displaymath}

  Therewith we may conclude that if $0 < \delta < \delta_0$ then:
  \begin{displaymath}
    \sup_{y \in A_\delta}
    \sup_{0 \leq s \leq \sqrt{\eta(\delta)}} \left\lvert
      f(X^\infty(\lambda_\delta^{(y)}(t)-s)) -
      f(X^\infty(t-s))
    \right\rvert < \epsilon .
  \end{displaymath}

  Since $\epsilon > 0$ was arbitrarily chosen, we may conclude that
  this quantity vanishes \qc as $\delta \to 0$.
\end{proof}

\begin{teorema}
  \label{teo:processo-markov}
  The K process is Markovian.
\end{teorema}

\begin{proof}
  Let us fix $m \geq 1$, $t_1 < t_2 < \ldots < t_{m+1}$ and $f_1,
  \ldots, f_{m+1} : \Nzb \to \R$ bounded continuous functions.

  As $X_\delta$ is Markovian:
  \begin{gather}
    \label{eq:markov-truncado}
    \E \left[
      f_{1}(X_\delta(t_{1}))
      \ldots
      f_{m}(X_\delta(t_{m}))
      f_{m+1}(X_\delta(t_{m+1}))
    \right] \\
    = \E \left[
      f_{1}(X_\delta(t_{1}))
      \ldots
      f_{m}(X_\delta(t_{m}))
      \Psi^\delta_{t_{m+1} - t_{m}} f_{m+1} (X_\delta(t_{m}))
    \right]\notag
  \end{gather}

  Remark \ref{obs:ponto-fixo-continuo} guarantee that $t_1, \ldots,
  t_{m+1}$ are \qc continuity points of the K process.  As convergence
  in the Skorohod topology guarantees pointwise convergence at 
  continuity points and each $f_i$ is continuous, we may conclude
  that $f_i(X_\delta(t_i)) \to f_i(X(t_i))$ \qc as $\delta \to 0$
  for every $i = 1, \ldots, m+1$.

  Using the dominated convergence theorem we conclude that the left
  hand side of \eqref{eq:markov-truncado} converges as $\delta \to 0$
  to:
  \begin{displaymath}
    \E \left[
      f_{1}(X(t_{1}))
      \ldots
      f_{m}(X(t_{m}))
      f_{m+1}(X(t_{m+1}))
    \right].
  \end{displaymath}

  Let us write the right hand side of \eqref{eq:markov-truncado}
  as:
  \begin{equation}
    \label{eq:markov-est-eps}
    \E \left[
      f_{1}(X_\delta(t_{1}))
      \ldots
      f_{m}(X_\delta(t_{m}))
      \Psi_{t_{m+1} - t_{m}} f_{m+1} (X_\delta(t_{m}))
    \right] + \epsilon_\delta.
  \end{equation}

  Using arguments analogous to those employed in the computation of the limit of the
  left hand side of \eqref{eq:markov-truncado}, together with
  Proposition \ref{prop:quase-feller}, we may show that the left hand
  side of \eqref{eq:markov-est-eps} converges as $\delta \to 0$ to:
  \begin{displaymath}
    \E \left[
      f_{1}(X(t_{1}))
      \ldots
      f_{m}(X(t_{m}))
      \Psi_{t_{m+1} - t_{m}} f_{m+1} (X(t_{m}))
    \right].
  \end{displaymath}

  We are left with showing that $|\epsilon_\delta| \to 0$ as $\delta \to 0$.
  For simplicity of notation, let us denote $s = t_{m+1} - t_m$, $t =
  t_m$ and $g = f_{m+1}$. Using the fact that $f_1, \ldots, f_m$ are
  bounded, we may write:
  \begin{displaymath}
    |\epsilon_\delta| \leq \text{const}
    \left\lvert \E \left[
        \Psi_{s}^\delta g (X_\delta(t)) -
        \Psi_{s} g (X_\delta(t))
      \right]
    \right\rvert.
  \end{displaymath}

  As the random variable inside the expected value is bounded, if we show
  that it converges almost surely to zero, then the result will follow
  from the dominated convergence theorem.

  If $V_t$ is as in Lemma \ref{lema:visitados-compacto}, note that
  there exists $t^{\prime} > 0$ such that $X_\delta(t) \in V_{t^\prime}$
  for every $\delta$.  Thus
  \begin{align*}
    \left\lvert \Psi_{s}^\delta g (X_\delta(t)) - \Psi_{s} g (X_\delta(t))
    \right\rvert
    &\leq \sup_{y \in V_{t^\prime}} \left\lvert \Psi_{s}^\delta g (y) - \Psi_{s} g (y)
    \right\rvert.
  \end{align*}

  Lemma \ref{lema:visitados-compacto} guarantees that the condition
  of Lemma \ref{lema:markov1} is valid for almost every
  $V_{t^\prime}$, so the above quantity vanishes \qc as $\delta
  \to 0$.
  
  Thus we can conclude that
  \begin{displaymath}
    \E \left[
      f_{1}(X(t_{1}))
      \ldots
      f_{m}(X(t_{m}))
      f_{m+1}(X(t_{m+1}))
    \right] \notag\\
    = \E \left[
      f_{1}(X(t_{1}))
      \ldots
      f_{m}(X(t_{m}))
      \Psi_{t_{m+1} - t_{m}} f_{m+1} (X(t_{m}))
    \right],
  \end{displaymath}
and the proof is complete.\qedhere
\end{proof}

\section{Transition probabilities}\label{trans}

\begin{definicao}
  \label{def:funcoes-transicao}
  For $x, y \in \Nzb$ and $t \geq 0$, let us define:
  \begin{equation}
    \label{eq:funcoes-transicao}
    p_{x y} (t) := \Prob(X^x(t) = y).
  \end{equation}
\end{definicao}

\begin{proposicao}
  \label{prop:trans-continuas}
  For every $x, y \in \Nzb$, $p_{x y} (\bullet)$ is a continuous
  function on $(0, \infty)$.
\end{proposicao}
\begin{proof}
  This is one of the items 
  of Theorem 1 from Section II of \cite{chung:67}.
\end{proof}

\begin{observacao}
  \label{obs:trans-facil}
  Note that Proposition \ref{prop:trans-continuas} states
  that the transition functions are continuous outside of the origin.
  It is easy to prove that $p_{x x} (t) \to 1$ as $t \searrow 0$ for every
  $x \in \Nz$; a harder problem arises when $x = \infty$.
\end{observacao}

\begin{proposicao}
  \label{prop:trans-c1}
  If $c > 0$, then:
  \begin{equation}
    \label{eq:trans-c1}
    \lim_{t \searrow 0} p_{\infty \infty} (t) = 1.
  \end{equation}
\end{proposicao}

\begin{proof}
  Consider the function:
  \begin{displaymath}
    \theta(t) := \int_0^t \ind \{ X^\infty (s) = \infty \} ds.
  \end{displaymath}

  Since $\Gamma^{\infty}(s) \to 0$ as $s \searrow 0$ (Remark
  \ref{obs:Gamma-cadlag}), then for a fixed $t > 0$ there exists an $s >
  0$ such that $\Prob(\Gamma^{\infty}(s) < t) > 0$.  By construction,
  on this event, $\theta(t) \geq c s$. From this we can conclude
  that $\E [\theta(t)] > 0$. Applying Fubini's theorem we obtain that:
  \begin{displaymath}
    0 < \E[\theta(t)] =  \int_0^t p_{\infty \infty} (s) ds
  \end{displaymath}

  Therefore the set $\{t > 0: p_{\infty \infty}(t) > 0 \}$ has a
  positive Lebesgue measure, so it is not empty. Let us fix an $s > 0$
  such that $p_{\infty \infty} (s) > 0$.

  Take an arbitrary sequence $(t_n)$, such that $t_n \searrow 0$ and
  $p_{\infty \infty} (t_n)$ converges to a real number $u$ as $n
  \to \infty$.  If we show that $u = 1$, then \eqref{eq:trans-c1} 
  follows.

  Using the Markov property and Proposition \ref{prop:trans-continuas},
  we can write:
  \begin{displaymath}
    p_{\infty\infty}(s) = \lim_{n \to \infty} p_{\infty \infty}(s+t_n) = \lim_{n \to
      \infty} \sum_{x \in \Nzb} p_{\infty x}(s) p_{x \infty} (t_n).
  \end{displaymath}

  By Remark \ref{obs:trans-facil}, if $x \in \Nz$ then $p_{x
    \infty}(t_n) \to 0$ as $n \to \infty$. Since each term in the sum
  can be bounded by $p_{\infty x} (s)$, which is summable, we can
  apply the dominated convergence theorem to conclude that:
    \begin{displaymath}
      p_{\infty\infty}(s) = p_{\infty \infty}(s)u.
    \end{displaymath}
    Since $p_{\infty \infty}(s) > 0$, we conclude that $u = 1$.
\end{proof}

\begin{proposicao}
  \label{prop:trans-c0}
  If $c = 0$, $x \in \Nzb$ and $t > 0$, then:
  \begin{displaymath}
    p_{x \infty} (t) = 0.
  \end{displaymath}
\end{proposicao}
\begin{proof}
  
  For $\delta \geq 0$ we define:
  \begin{align*}
    \theta_\delta(t) := \int_0^t \ind \left\{ X^\infty_0(s) \not\in
      \Nz_\delta \right\} d s
  \end{align*}

  Note that $\theta_\delta(t)$ is monotonic in $\delta$. Thus,
  if $\Xi$ is the right-continuous inverse for $\Gamma^\infty$ then:
  \begin{align*}
    \theta_0(t) \leq \theta_\delta(t) = \sum_{x \not\in \Nz_\delta}
    \sum_{i=1}^{\Xi(t)} \gamma_x T^x_i.
  \end{align*}

  As $\delta \searrow 0$ we are summing over the tail of a convergent
  series, thus $\lim_{\delta \searrow 0} \theta_\delta(t) = 0$
  \qc, so $\theta_0(t) =  0$ \qc and $\E[\theta_0(t)] = 0$. Using
  Fubini's theorem we can write
  \begin{align*}
    0 = \E\left[ \int_0^t \ind \left\{ X^\infty_0(s) = \infty
      \right\} ds \right]
    = \int_0^t \Prob \left\{ X^\infty_0(s) = \infty
    \right\} ds
    = \int_0^t p_{\infty \infty} (s) ds,
  \end{align*}
  from which we conclude that $p_{\infty \infty} (t) = 0$ for Lebesgue
  almost all point $t > 0$. Using the continuity of $p_{\infty
    \infty}$ (Proposition \ref{prop:trans-continuas}) and Remark
  \ref{obs:reinicia-infinito} we conclude this proposition.
\end{proof}

\paragraph{Transition rates}
  Let us consider the transition rates of the K process, given for $x, y \in
  \Nzb$ by:
  \begin{displaymath}
    q_{x y} = \lim_{t \searrow 0} \frac{p_{x y}(t) - \ind\{x = y\}}{t}.
  \end{displaymath}
  Most of the remainder of this section will be dedicated to proving that
  these limits exist and computing their values for different choices
  of $x$ and $y$. At the very end we will use that information to 
  give the invariant measure for the K process.

\begin{proposicao}
  \label{prop:taxa-x-y}
  For every $x \in \Nz$ and $A \subset \Nz\setminus\{x\}$
  such that $\sum_{y \in A} \lambda_y < \infty$:
  \begin{equation}
    \label{eq:taxa-x-y}
    \lim_{t \searrow 0} \frac{\Prob(X^x(t) \in A)}{t} = 0.
  \end{equation}
  As a particular case, taking $A = \{y\}$, $y \in \Nz\setminus \{x\}$, we have
  that $q_{x y} = 0$.
\end{proposicao}
\begin{proof}

  Let us denote by $\sigma_A = \inf_{y \in A} \sigma_1^y$; this is an
  exponential random variable of rate $\sum_{y\in A} \lambda_y$.

  Using Remark \ref{obs:reinicia-infinito}, we can write:
  \begin{align*}
    \frac{\Prob (X^x(t) \in  A)}{t}
    &= \frac{1}{t} \int_{0}^{t} \Prob( X^\infty(t-s) \in A ) \frac{1}{\gamma_x}
    e^{-\frac{s}{\gamma_x}} ds \\
    &\leq \frac{1}{t \gamma_x} \int_{0}^{t} \Prob( \Gamma^\infty(\sigma_A-) \leq t-s) ds \\
    &= \frac{1}{t \gamma_x} \int_{0}^{t} \Prob( \Gamma^\infty(\sigma_A-) \leq s) ds.
  \end{align*}

  Note that $\Prob( \Gamma^\infty(\sigma_A-) = 0) = 0$, so the
  probability inside the integral vanishes as $s \searrow 0$.
\end{proof}

\begin{proposicao}
  \label{prop:taxa-x-x}
  For every $x \in \Nz$:
  \begin{equation}
    \label{eq:taxa-x-x}
    q_{x x} = -\frac{1}{\gamma_x}.
  \end{equation}
\end{proposicao}
\begin{proof}

  Using Remark \ref{obs:reinicia-infinito}, we can write:
  \begin{align*}
    \frac{p_{x x} (t) - 1}{t}
    &= \frac{\Prob( \gamma_x T_0^x > t)-1}{t} + 
    \frac{1}{t}\int_{0}^{t} \Prob( X^\infty(t-s) = x)
    \frac{1}{\gamma_x} e^{-\frac{s}{\gamma_x}} ds\\ 
    &= \frac{e^{-\frac{t}{\gamma_x}} - 1}{t} + 
    \frac{1}{t} \int_{0}^{t} p_{\infty x}(t-s) \frac{1}{\gamma_x}
    e^{-\frac{s}{\gamma_x}} ds\\ 
  \end{align*}

  The first term of the latter sum converges to $-\frac{1}{\gamma_x}$ as $t \searrow 0$.
  Using arguments analogous to the ones used in the proof of
  Proposition \ref{prop:taxa-x-y}, we can show that the second term of
  this sum vanishes as $t \searrow 0$.
\end{proof}

\begin{proposicao}
  \label{prop:taxa-x-inf}
  For $x \in \Nz$:
  \begin{displaymath}
    q_{x \infty} = \begin{cases}
      \frac{1}{\gamma_x} & \textrm{ if } c > 0 \\
      0 & \textrm{ if } c = 0 . \\
    \end{cases}
  \end{displaymath}
\end{proposicao}
\begin{proof}
  The case $c = 0$ is trivial using Proposition
  \ref{prop:trans-c0}. So let us suppose $c > 0$.
  \begin{align*}
    \frac{p_{x \infty} (t)}{t}
    &= \frac{1}{t} \int_0^t \Prob(X^{\infty} (t-s) = \infty)
    \frac{1}{\gamma_x} e^{-\frac{s}{\gamma_x}} ds\\
    &= \frac{1}{t} \int_0^t p_{\infty \infty} (t-s)
    \frac{1}{\gamma_x}e^{-\frac{s}{\gamma_x}} ds\\
    &= \frac{e^{-t}}{\gamma_x} \frac{1}{t} \int_0^t p_{\infty \infty}
    (s) e^{\frac{s}{\gamma_x}} ds .
  \end{align*}

  Proposition \ref{prop:trans-c1} guarantees that 
  $p_{\infty \infty}(t) \to 1$ as $t \searrow 0$. So the above
  quantity converges to $\frac{1}{\gamma_x}$ as $t \searrow 0$.
\end{proof}

\begin{proposicao}
  \label{prop:taxa-inf-x}
  For $x \in \Nz$:
  \begin{displaymath}
    q_{\infty x} = \begin{cases}
      \frac{\lambda_x}{c} & \textrm{ if } c > 0 \\
      \infty & \textrm{ if } c = 0 \\
    \end{cases}
  \end{displaymath}
\end{proposicao}
\begin{proof}
  \begin{align}
    \frac{p_{\infty x}(t)}{t} &=
    \frac{1}{t} \Prob \left( \bigcup_{i =1}^{\infty} \left\{
        \Gamma^\infty (\sigma^x_i -) \leq t <
        \Gamma^\infty(\sigma^x_i) \right\}
    \right) \notag \\
    &= \frac{\Prob \left( \Gamma^\infty (\sigma^x_1 -) \leq t <
      \Gamma^\infty(\sigma^x_1) \right)}{t}
  + \frac{1}{t} \Prob \left( \bigcup_{i =
        2}^{\infty} \left\{ \Gamma^\infty (\sigma^x_i -) \leq t <
        \Gamma^\infty(\sigma^x_i) \right\} \right) \notag \notag \\
    \label{erros_taxa_inf}
    &= \frac{\Prob \left( \Gamma^\infty (\sigma^x_1 -) \leq t \right)}{t} -
    \frac{\Prob \left( \Gamma^\infty (\sigma^x_1) \leq t \right)}{t} +
    \frac{1}{t} \Prob \left( \bigcup_{i =
        2}^{\infty} \left\{ \Gamma^\infty (\sigma^x_i -) \leq t <
        \Gamma^\infty(\sigma^x_i) \right\} \right)
  \end{align}

  Working with the second term of this sum, we obtain:
  \begin{align*}
    \frac{\Prob (\Gamma^\infty (\sigma^x_1) \leq t)}{t}
    &= \frac{1}{t} \Prob \left(
      \gamma_x T^x_1 + 
      \sum_{y \neq x} \sum_{i = 1}^{N_y (\sigma^x_1)} \gamma_y T^y_i +
      c\sigma^x_1
      \leq t
    \right) \\
    &\leq \frac{1}{t} \Prob \left(
      \gamma_x T^x_1 + 
      \sum_{y \neq x} \sum_{i = 1}^{N_y (\sigma^x_1)} \gamma_y T^y_i
      \leq t
    \right)\\
    &= \frac{1}{t} \int_0^t \Prob \left(
      \sum_{y \neq x} \sum_{i = 1}^{N_y (\sigma^x_1)} \gamma_y T^y_i
      \leq t - s
    \right) \frac{1}{\gamma_x} e^{-\frac{s}{\gamma_x}} ds\\
    &\leq \frac{1}{\gamma_x t} \int_0^t \Prob \left(
      \sum_{y \neq x} \sum_{i = 1}^{N_y (\sigma^x_1)} \gamma_y T^y_i
      \leq t - s
    \right) ds\\
    &= \frac{1}{\gamma_x t} \int_0^t \Prob \left(
      \sum_{y \neq x} \sum_{i = 1}^{N_y (\sigma^x_1)} \gamma_y T^y_i
      \leq s
    \right) ds\\
  \end{align*}

  The probability inside the last integral vanishes as $s \to 0$,
  from which we can conclude that the second term of
  \eqref{erros_taxa_inf} vanishes as $t \searrow 0$.

  Note that the event in the third term of \eqref{erros_taxa_inf}
  is contained in the event of the second term, so the third term also
  vanishes as $t \searrow 0$.

  Now we are left with computing the limit of the first term. We are going
  to do this using Karamata's Tauberian Theorem, that relates the
  behavior of the distribution function of a positive random variable
  near the origin with the behavior of its Laplace transform near
  infinity. We will follow the statement of this theorem given in
  \cite{fellerv2} ({Theorem XIII.5.1}).

  Conditioning in the value of $\sigma_1^x$ and using the independence
  of the Poisson processes, we can compute the Laplace transform
  $\phi$ of $\Gamma^\infty(\sigma^x_1-)$ by:
  \begin{displaymath}
    \phi (\beta) := \E \left[ e^{-\beta \Gamma^\infty (\sigma^x_1-)}  \right] 
    = \lambda_x \left( \lambda_x + \beta c + \beta \sum_{y \neq x}
      \frac{\lambda_y \gamma_y}{1 + \beta\gamma_y} \right)^{-1}.
  \end{displaymath}
  
  Let us first treat the case $c > 0$. Note that in this case:
  \begin{align*}
      \lim_{\alpha \to \infty} \frac{\phi(\alpha \beta)}{\phi
        (\alpha)} &= \frac{1}{\beta}.
  \end{align*}

  So the condition of the Tauberian theorem is verified, from which we
  obtain that, when $c > 0$:
  \begin{align*}
    \lim_{t \searrow 0} \frac{P( \Gamma^\infty(\sigma^x_1-) \leq
      t)}{\phi(\frac{1}{t})} = 1
  \end{align*}

  Note that $\frac{\phi(\frac{1}{t})}{t} \to \frac{\lambda_x}{c}$ as
  $t \searrow 0$, from which we conclude the case $c > 0$.

  To treat the case $c = 0$, note that our construction allows for a
  coupling of several K processes, with different values of $c$, using
  the same Poisson processes and exponential random
  variables. So let us attach an index $c$ in $\Gamma^y_c$ to denote
  what is the value of $c$ to which we are referring.

  Note that $\Gamma^y_c$ is increasing in $c$, so $\Prob (
  \Gamma^\infty_c(\sigma^x_1-) \leq t)$ is monotonic in $c$.
  Finally we can conclude that, for every $c > 0$:
  \begin{align*}
    \liminf_{t \searrow 0} \frac{\Prob ( \Gamma^\infty_0(\sigma^x_1-) \leq
      t)}{t} &\geq \liminf_{t \searrow 0} \frac{\Prob (
      \Gamma^\infty_c(\sigma^x_1-) \leq t)}{t}
    = \frac{\lambda_x}{c}.
  \end{align*}

  Taking small values of $c > 0$, we conclude that $\frac{\Prob (
    \Gamma^\infty_0(\sigma^x_1-))}{t} \xrightarrow{t \searrow 0}
  \infty$.
\end{proof}

\begin{proposicao}
  \label{prop:taxa-inf-inf}
  \begin{displaymath}
    q_{\infty \infty} = -\infty.
  \end{displaymath}
\end{proposicao}
\begin{proof}
  The case $c = 0$ is trivial using Proposition
  \ref{prop:trans-c0}. So let us suppose $c > 0$. 

  Taking $\{x_1, x_2, \ldots\}$ an enumeration of $\Nz$, we can write:
  \begin{equation}
    \label{eq:taxa-inf-inf}
    \frac{p_{\infty \infty} (t) - 1}{t} 
    = - \sum_{x \in \Nz} \frac{p_{\infty x} (t)}{t}
    \leq - \sum_{i = 1}^n \frac{p_{\infty x_i} (t)}{t}.
  \end{equation}
  
  The last term converges to $-\sum_{i = 1}^n \frac{\lambda_{x_i}}{c}$ as
  $t \searrow 0$ by Proposition \ref{prop:taxa-inf-x}, and this 
  goes to $-\infty$ as $n \to \infty$ because of the first condition
  in~\eqref{eq:soma-lambda-gamma}.
\end{proof}

\begin{observacao}
  \label{obs:reuter}
  The above results on the transition rates of the K process, adjoined to the Markov property,
and the uniqueness result of~\cite{reuter:69}, imply that in the $c>0$ case, the K process is
a version of the process introduced in the latter reference. In the notation of that reference,
we have the following correspondence with the notation of the present paper:
$$a_x=1/\gamma_x,\quad b_x=\lambda_x/c.$$
\end{observacao}

\begin{teorema}
  \label{teo:dist-invariante}
  The K process has a unique invariant probability measure, given by:
  \begin{equation}
    \label{eq:dist-invariante}
    \pi(x) = \begin{cases}
      \frac{\lambda_x \gamma_x}{c + \sum_{y \in \Nz} \lambda_y
        \gamma_y} & \text{if } x \in \Nz\\
      \frac{c}{c + \sum_{y \in \Nz} \lambda_y
        \gamma_y} & \text{if } x = \infty
    \end{cases}
  \end{equation}
\end{teorema}
\begin{proof} (Sketchy)

  Let us first prove the existence and uniqueness of the invariant
  measure, using a standard argument. For a fixed $h > 0$ consider 
  the discrete time Markov Chain $(Y^h_n)_{n \in \N}$ given by:
  \begin{displaymath}
    Y^h_n := X(n h).
  \end{displaymath}

  Note that this discrete time Markov Chain is irreducible in the case
  $c > 0$ and that there is only one closed class of communication, namely
  $\Nz$, in the case $c = 0$.  Furthermore, in both cases, every
  state $x \in \Nz$ is positive recurrent.

  This implies that, for every choice of $h > 0$, there exists a unique
  invariant probability measure $\mu_h$ for $(Y^h_n)_n$. The Markov
  property yields that this probability measure is the same for every
  rational choice of $h$.

  Finally using the continuity of the transition functions
  (Proposition \ref{prop:trans-continuas}) we conclude that $\mu_h$
  is the same for every choice of $h$. It follows that this is the unique
  invariant probability measure of the K process.

  Let us thus drop the sub-index $h$ and refer to this probability
  measure simply as $\mu$.  For computing it in the case $c > 0$, note
  that it needs to satisfy:
  \begin{displaymath}
    \frac{\mu(y)}{t} = \sum_{x \in \Nzb} \mu(x) \frac{p_{x y}(t)}{t}
  \end{displaymath}
  for every $t>0$. Taking the limit $t \searrow 0$, we conclude from
  the above results about the transition rates of the process that for
  $y\in \Nz$:
  \begin{displaymath}
    \mu(\infty) q_{\infty y} + \mu(y) q_{y y} = 0.
  \end{displaymath}

  Using Propositions \ref{prop:taxa-x-x} and \ref{prop:taxa-inf-x}, we
  conclude that the only solution for this system of equations that is
  a probability measure is $\pi$ given in \eqref{eq:dist-invariante}.

  Since $p_{x y} (t)$ varies continually as a function of $c$ for every fixed 
  $x, y \in \Nz$ and $t > 0$, as can be shown via a straightforward coupling 
  argument, we may obtain the case $c = 0$ by taking the limit $c\to
  0$.
\end{proof}

\section{Infinitesimal Generator}\label{ig}

As pointed out in the introduction, the case of positive $c$ was treated in~\cite{reuter:69}.
We will restrict our attention in this section to the case where $c = 0$.

We will denote by $\CC_0$ the set of bounded continuous functions. We will show below that
the following subset of $\CC_0$ is a core for the generator of the K process.
\begin{equation}
  \label{eq:classe-D}
  \DD = \left\{
    f \in \CC_0 :
    \begin{array}{l}
      \sum_{x \in \Nzb} |f(x)-f(\infty)|\lambda_x < \infty,\,\,
      \sum_{x \in \Nzb} (f(x)-f(\infty))\lambda_x = 0,\\
     \mbox{}\hspace{2cm}\lim_{x \to \infty} \frac{f(x) - f(\infty)}{\gamma_{x}}
      \text{ exists } 
    \end{array}
  \right\}
\end{equation}
Let us recall that in $(\Nzb,d)$, $x\to\infty$ is equivalent to $\gamma_x\to0$.

\begin{teorema}
  \label{teo:gerador-x}
  For every $x \in \Nz$, and $h \in \DD$ we have that:
  \begin{equation}
    \label{eq:gerador-x}
    \lim_{t \searrow 0} \frac{\E \left[ h(X^x(t))\right] - h(x)}{t} =
    \frac{h(\infty) - h(x)}{\gamma_x}.
  \end{equation}
\end{teorema}
\begin{proof}
  We will show that, for any fixed $\epsilon > 0$:
  \begin{gather}
    \label{eq:gerador-x-sup}
    \limsup_{t \searrow 0} \frac{\E \left[ h(X^x(t))\right] - h(x)}{t} \leq
    \frac{h(\infty) - h(x) + \epsilon}{\gamma_x}\\
    \label{eq:gerador-x-inf}
    \liminf_{t \searrow 0} \frac{\E \left[ h(X^x(t))\right] - h(x)}{t} \geq
    \frac{h(\infty) - h(x) - \epsilon}{\gamma_x},
  \end{gather}
  from these inequalities, \eqref{eq:gerador-x} follows
  immediately. We will only show \eqref{eq:gerador-x-sup}, since
  \eqref{eq:gerador-x-inf} is analogous.

  Since $h \in \DD \subset \CC_0$, we can take $\delta \in (0,
  \gamma_x)$ and $H > 0$ such that $|h(y)-h(\infty)| < \epsilon$
  whenever $\gamma_y < \delta$ and $\sup_{y\in \Nzb} |h(y)| < H$.
  
  \begin{align}
    \label{eq:gerador-x-1}
    \frac{\E \left[ h(X^x(t))\right] - h(x)}{t} &=
    \sum_{y \not\in \Nz_\delta} (h(y) - h(x)) \frac{p_{x
        y}(t)}{t} +
    \sum_{y \in \Nz_\delta\setminus\{x\}} (h(y) - h(x)) \frac{p_{x
        y}(t)}{t}
  \end{align}

  The second term can be dominated by $2 H \frac{1}{t}\Prob (X^x(t) \in
  \Nz_\delta\setminus\{x\})$. Since $\sum_{y \in \Nz_\delta} \lambda_y
  < \infty$ (by Remark \ref{obs:restric-truncados}), we can use
  Proposition \ref{prop:taxa-x-y} to get that this quantity vanishes as $t \searrow 0$.

  The first term of \eqref{eq:gerador-x-1} can be dominated by:
  \begin{align*}
    (h(\infty) - h(x) + \epsilon) \frac{\Prob(X^x (t) \not\in
      \Nz_\delta)}{t}
    &=
    (h(\infty) - h(x) + \epsilon) \frac{1 - \Prob(X^x (t) \in
      \Nz_\delta)}{t}\\
    &=
    (h(\infty) - h(x) + \epsilon)
    \left[
      \frac{1 - p_{x x}(t)}{t} -
      \frac{\Prob(X^x (t) \in \Nz_\delta\setminus\{x\})}{t}
    \right].
  \end{align*}

  We can again use Proposition \ref{prop:taxa-x-y} to show that the
  second term inside the brackets converges to zero, while the first
  term converges to $1/\gamma_x$ by Proposition
  \ref{prop:taxa-x-x}.
\end{proof}

\begin{definicao}
  \label{def:visitas}
  For $x \in \Nz$, let us denote by ``$X(t) = x$ first visit'' the
  event that ``at time $t$, $X$ is at its first visit to $x$'';
  more precisely, making $L_0^x\equiv0$ and $H_i^x=\inf\{t>L_{i-1}^x:X(t)=x\}$, 
  $L_i^x=\inf\{t>H_{i}^x:X(t)\ne x\}$, $i\geq1$, then $\{X(t) = x
  \text{ first visit}\} = \{ X(t) = x,\,H_2^x>t  \}$.

  Conversely, ``X(t) = x not first visit'' is the event $\{X(t) = x\}
  \setminus \{X(t) = x \text{ first visit}\}$.
\end{definicao}

\begin{lema}
  \label{lema:gerador-2visitas}
  If $h \in \DD$ and $\sup_{x\in\Nz}\lambda_x < \infty$ then:
  \begin{equation}
    \label{eq:gerador-2visitas}
    \lim_{t \searrow 0} \frac{1}{t} \sum_{x \in \Nz} |h(x)-h(\infty)|
    \Prob\left(
      X^\infty(t) = x \text{ not first visit}
    \right) = 0
  \end{equation}
\end{lema}

\begin{proof}

  Throughout this proof we will always assume that the process started
  at $\infty$ and omit this index of our notations. We will also
  assume, without loss of generality, that $h(\infty) = 0$.
  For $x\in\Nz$ and $t\geq0$, let
  \begin{equation}
    \label{eq:Gx}
   \Gamma^{<x>}(t)=\sum_{z \in \Nz\setminus\{x\}} \sum_{i = 1}^{N_z(t)} \gamma_z T^z_i.
  \end{equation}

  For arbitrary $x \in \Nz$ and $t > 0$:
  \begin{align*}
    \Prob\left(\Gamma^{<x>}(\sigma^x_2) \leq t \right)
    &= \Prob\left(
      e^{-\frac{\Gamma^{<x>}(\sigma^x_2)}{t}} \geq e^{-1} \right)
    \leq e \,\E \exp\left\{
      -\frac{\Gamma^{<x>}(\sigma^x_2)}{t}
    \right\}.
  \end{align*}

  Since $\sigma^x_2$ is a sum of two independent exponential random
  variables and $\Gamma$ have stationary and independents increments,
  we have that:
  \begin{align*}
    E \left[ \exp\left\{
        -\frac{\Gamma^{<x>}(\sigma^x_2)}{t}
      \right\} \right]
    = E \left[ \exp\left\{
        -\frac{\Gamma^{<x>}(\sigma^x_1)}{t}
      \right\} \right]^2\\
    = \left(
      1 + \frac{1}{\lambda_x} \sum_{y \neq x} \frac{\lambda_y
        \gamma_y}{t + \gamma_y}
    \right)^{-2}
    \leq \left(
      \frac{1}{\lambda_x} \sum_{y \in \Nz} \frac{\lambda_y
        \gamma_y}{t + \gamma_y}
    \right)^{-2}
  \end{align*}

  The event $\{X^\infty(t) = x \text{ not first visit}\}$ is
  contained in the event $\{ \Gamma^{<x>}(\sigma^x_2) \leq t\} \cap \{
  \gamma_x T^x_1 \leq t\}$, which are independent. Thus, since
  $\sup_x \lambda_x < \infty$:
  \begin{align}
    \label{eq:gerador-2visitas-inter}
    \frac{1}{t} \sum_{x \in \Nz} |h(x)|
    \Prob\left(
      X^\infty(t) = x \text{ not first visit}
    \right) &\leq \textrm{const}
    \left(
      \sum_{y \in \Nz} \frac{\lambda_y
        \gamma_y}{t + \gamma_y}
    \right)^{-2}
    \sum_{x \in \Nz} |h(x)| \lambda_x 
    \frac{1 - e^{-t/\gamma_x}}{t}
  \end{align}

  The right hand side of this inequality (except for the constant) can
  be written as:
  \begin{align}
    \label{eq:gerador-2visitas-1}
    &
    \left(
      \sum_{y \in \Nz} \frac{\lambda_y
        \gamma_y}{t + \gamma_y}
    \right)^{-2}
    \sum_{x \in \Nz_t} |h(x)| \lambda_x 
    \frac{1 - e^{-t/\gamma_x}}{t}\\
    \label{eq:gerador-2visitas-2}
    +&
    \left(
      \sum_{y \in \Nz} \frac{\lambda_y
        \gamma_y}{t + \gamma_y}
    \right)^{-2}
    \sum_{x \not\in \Nz_t} |h(x)| \lambda_x 
    \frac{1 - e^{-t/\gamma_x}}{t}
  \end{align}

  Let us denote by $s_t = \sum_{x \in \Nz_t} \lambda_x$. Note that
  $s_t < \infty$ for every $t > 0$ but $s_t \to \infty$ as $t \searrow
  0$.

  To control \eqref{eq:gerador-2visitas-1}, we will use that
  $(1-e^{-t/\gamma_x})/t \leq 1/\gamma_x$ and that $|h(x)|/\gamma_x$
  is bounded by a constant (since $h \in \DD$).
  With these considerations we can bound \eqref{eq:gerador-2visitas-1}
  by:
  \begin{align*}
    \left(
      \sum_{y \in \Nz_t} \frac{\lambda_y
        \gamma_y}{t + \gamma_y}
    \right)^{-2}
    \sum_{x \in \Nz_t} \frac{|h(x)|}{\gamma_x} \lambda_x 
    &\leq
    \left(
      \sum_{y \in \Nz_t} \frac{\lambda_y
        \gamma_y}{2 \gamma_y}
    \right)^{-2}
    \sum_{x \in \Nz_t} \textrm{const } \lambda_x\\
    &= \textrm{const } \frac{s_t}{s_t^2} = 
    \textrm{const } \frac{1}{s_t}
    \xrightarrow{t\searrow 0} 0 
  \end{align*}

  To control \eqref{eq:gerador-2visitas-2}, we will use that
  $1-e^{-t/\gamma_x} \leq 1$ and $|h(x)| \leq \textrm{const}\,
  \gamma_x$. Thus we can bound \eqref{eq:gerador-2visitas-2}
  (except for a constant factor) by:
  \begin{align*}
    \left(
      \sum_{y \in \Nz} \frac{\lambda_y
        \gamma_y}{t + \gamma_y}
    \right)^{-2}
    \sum_{x \not\in \Nz_t}  \frac{\lambda_x \gamma_x}{t}
  \end{align*}

  If the sum in $x$ is bounded for every value of $t$, then this
  expression vanishes as $t\searrow 0$, otherwise we can bound
  this expression as:
  \begin{displaymath}
    \left(
      \sum_{y \in \Nz} \frac{\lambda_y
        \gamma_y}{2 t}
    \right)^{-2}
    \sum_{x \not\in \Nz_t}  \frac{\lambda_x \gamma_x}{t} = 
    4 \left( \sum_{y \in \Nz} \frac{\lambda_y \gamma_y}{t}
    \right)^{-1}
    \xrightarrow{t\searrow 0} 0
    \qedhere
  \end{displaymath}
\end{proof}

\begin{teorema}
  \label{teo:gerador-inf}
   Let us assume that $\sup_{x \in \Nz} \lambda_x <
  \infty$. Then for every $h \in \DD$ we have that:
  \begin{equation}
    \label{eq:gerador-inf}
    \lim_{t \searrow 0} \frac{\E \left[ h(X^\infty(t))\right] -
      h(\infty)}{t} =
    \lim_{x \to \infty}
    \frac{h(\infty) - h(x)}{\gamma_x}.
  \end{equation}
\end{teorema}

\begin{proof}
  Let us denote by $\bar{h}(x) = h(x) - h(\infty)$ and let $L =
  \lim_{x \to \infty} \frac{\bar{h}(x)}{\gamma_x}$. Without loss of
  generality let us suppose that $L \geq 0$.

  \begin{align}
    \label{eq:gerador-inf-1}
    \frac{E \left[ h(X^\infty(t))\right] - h(\infty)}{t}
    &= \frac{1}{t}
    \sum_{x \in \Nz} \bar{h}(x) \Prob(X(t) = x \text{ first visit})
    + \frac{1}{t}
    \sum_{x \in \Nz} \bar{h}(x) \Prob(X(t) = x \text{ not first visit})
  \end{align}
  (recall Definition \ref{def:visitas}).

  Note that the second term of right hand side of \eqref{eq:gerador-inf-1} vanishes as
  $t \searrow 0$ by Lemma \ref{lema:gerador-2visitas}.
  Since $\sum_x \bar{h}(x) \lambda_x = 0$, if we fix an arbitrary $a
  \in \Nz$, then we can write the first term as:
  \begin{align}
    &\notag
    \frac{1}{t} \sum_{x \in \Nz}
    \bar{h}(x) \Prob(X(t) = x \text{ first visit})
    - \frac{1}{t} \sum_{x \in \Nz}
    \bar{h}(x) \lambda_x \frac{\Prob(X(t) = a \text{ first
        visit})}{\lambda_a} \\
    \notag
    =&
    \frac{1}{t} \sum_{x \in \Nz\setminus \{a\}}
    \frac{\bar{h}(x)}{\lambda_a} \left[
      \lambda_a \Prob(X(t) = x \text{ first visit}) -
      \lambda_x \Prob(X(t) = a \text{ first visit})
    \right]\\
    \label{eq:gerador-inf-2}
    =&
    \frac{1}{t} \sum_{x \in \Nz\setminus \{a\}}
    \frac{\bar{h}(x)}{\lambda_a} \left[
      \lambda_a \Prob(X(t) = x \text{ first time}, \sigma^x_1 < \sigma^a_1) -
      \lambda_x \Prob(X(t) = a \text{ first visit}, \sigma^a_1 < \sigma^x_1)
    \right]\\
    \label{eq:gerador-inf-3}
    +&
    \frac{1}{t} \sum_{x \in \Nz\setminus \{a\}}
    \frac{\bar{h}(x)}{\lambda_a} \left[
      \lambda_a \Prob(X(t) = x \text{ first visit}, \sigma^x_1 > \sigma^a_1) -
      \lambda_x \Prob(X(t) = a \text{ first visit}, \sigma^a_1 > \sigma^x_1)
    \right].
  \end{align}
  
  We can bound the absolute value of \eqref{eq:gerador-inf-3} by:
  \begin{align}
    \label{eq:gerador-inf-erro1-1}
    &\frac{1}{t} \sum_{x \in \Nz\setminus \{a\}}
    |\bar{h}(x)| \Prob(X(t) = x \text{ first visit}, \sigma^x_1 >
    \sigma^a_1)\\
    \label{eq:gerador-inf-erro1-2}
    +& \frac{1}{\lambda_a}
    \frac{1}{t} \sum_{x \in \Nz\setminus \{a\}}  |\bar{h}(x)| 
    \lambda_x \Prob(X(t) = a \text{ first visit}, \sigma^a_1 > \sigma^x_1)
  \end{align}

  Analogously to the proof of Lemma \ref{lema:gerador-2visitas}, we
  can bound \eqref{eq:gerador-inf-erro1-1} by:
  \begin{align}\notag
    \frac{1}{t} \sum_{x\in\Nz\setminus\{a\}}
    |\bar{h}(x)| \Prob(\Gamma^{<x, a>}(\sigma^x_1) \leq t, \gamma_a
    T^1_a \leq t) &\leq
    \text{const}  \sum_{x\in\Nz\setminus\{a\}}  |\bar{h}(x)|
    \frac{1-e^{-\frac{t-s}{\gamma_a}}}{t} \left(
      \frac{1}{\lambda_x} \sum_{y\in \Nz\setminus\{a\}}
      \frac{\lambda_y \gamma_y}{t+\gamma_y}
    \right)^{-1}\\\label{eq:gxa}
    &\leq
    \text{const}  \sum_{x\in\Nz\setminus\{a\}}  |\bar{h}(x)|\lambda_x
    \left(
      \sum_{y\in \Nz\setminus\{a\}}
      \frac{\lambda_y \gamma_y}{t+\gamma_y}
    \right)^{-1},
  \end{align}
  where
  \begin{equation}
    \label{eq:Gxa}
   \Gamma^{<x,a>}(t)=\sum_{z \in \Nz\setminus\{x,a\}} \sum_{i = 1}^{N_z(t)} \gamma_z T^z_i.
  \end{equation}
  The quantity on the right hand side of~(\ref{eq:gxa}) vanishes as $t
  \searrow 0$ by the dominated convergence theorem.

  Given that $\sigma^a_1 > \sigma^x_1$, then $\sigma^a_1$ has the
  same distribution as a sum of two independent exponential random
  variables, of rates $\lambda_x+\lambda_a$ and $\lambda_a$. Using
  this fact we can bound \eqref{eq:gerador-inf-erro1-2} by:
  \begin{align*}
    &\frac{1}{\lambda_a} \frac{1}{t}
    \sum_{x \in \Nz\setminus \{a\}}  |\bar{h}(x)| \lambda_x
    \Prob(
    \Gamma^{<x, a>}(\sigma^a_1) \leq t,
    \gamma_x T^x_1 \leq t, \sigma^a_1 > \sigma^x_1)
    \\\leq& \text{ const}
    \sum_{x \in \Nz\setminus \{a\}}  |\bar{h}(x)| \lambda_x
    \frac{1 - e^{-t/\gamma_x}}{t}
    \frac{\lambda_x}{\lambda_a+\lambda_x}
    \left(
      \frac{1}{\lambda_x+\lambda_a} \sum_{y\in \Nz\setminus\{a\}}
      \frac{\lambda_y \gamma_y}{t+\gamma_y}
    \right)^{-1}
    \left(
      \frac{1}{\lambda_a} \sum_{y\in \Nz\setminus\{a\}}
      \frac{\lambda_y \gamma_y}{t+\gamma_y}
    \right)^{-1}\\
    \leq &\text{ const}
    \left(
      \sum_{y\in \Nz\setminus\{a\}}
      \frac{\lambda_y \gamma_y}{t+\gamma_y}
    \right)^{-2}
    \sum_{x \in \Nz\setminus \{a\}}  |\bar{h}(x)| \lambda_x
    \frac{1 - e^{-t/\gamma_x}}{t}.
  \end{align*}
  Disregarding constants, the latter quantity is exactly \eqref{eq:gerador-2visitas-inter}, 
  which showed up at an intermediate step of the proof of Lemma
  \ref{lema:gerador-2visitas}. Following that proof, we conclude
  that \eqref{eq:gerador-inf-erro1-2} vanishes as $t\searrow 0$.

  Suppose that $S_x$ is a exponential random variable with rate
  $\lambda_a + \lambda_x$, independent of the process.  If we call the
  distribution function and the density of $\Gamma^{<x,a>}(S_x)$ (as
  in \eqref{eq:Gxa}) by
  $F_x$ and $f_x$ respectively, then we can write
  \eqref{eq:gerador-inf-2} as:
  \begin{align}
    \notag&
    \frac{1}{t} \sum_{x \in \Nz\setminus \{a\}}
    \bar{h}(x)\frac{\lambda_x}{\lambda_a+\lambda_x} 
    \int_0^t f_x (s) \left[
      e^{-\frac{t-s}{\gamma_x}} -
      e^{-\frac{t-s}{\gamma_a}}
    \right] d s\\
    \label{eq:gerador-inf-4}
    =&
    - \frac{1}{t} \sum_{x \in \Nz\setminus \{a\}}
    \bar{h}(x)\frac{\lambda_x}{\lambda_a+\lambda_x} 
    \int_0^t f_x (s) \left[
      1 - e^{-\frac{t-s}{\gamma_x}}
    \right] d s\\
    \label{eq:gerador-inf-5}
    &+
    \frac{1}{t} \sum_{x \in \Nz\setminus \{a\}}
    \bar{h}(x)\frac{\lambda_x}{\lambda_a+\lambda_x} 
    \int_0^t f_x (s) \left[
      1 - e^{-\frac{t-s}{\gamma_a}}
    \right] d s
  \end{align}

  The absolute value of \eqref{eq:gerador-inf-5} can be bounded by:
  \begin{align*}
    \frac{1}{\lambda_a}
    \sum_{x \in \Nz\setminus \{a\}}
    |\bar{h}(x)|\lambda_x
    \int_0^t f_x (s) \frac{1 - e^{-\frac{t-s}{\gamma_a}}}{t} d s
    &\leq
   \frac{1}{\lambda_a \gamma_a}
    \sum_{x \in \Nz\setminus \{a\}}
    |\bar{h}(x)|\lambda_x F_x(t).
  \end{align*}

  This vanishes as $t \searrow 0$ by the dominated convergence theorem.

  For an $\epsilon > 0$ fixated, take $\delta > 0$ such that
  $|\frac{\bar{h}(x)}{\gamma_x} - L| < \epsilon$ whenever $\gamma_x < \delta$.
  Integrating \eqref{eq:gerador-inf-4} by parts, we obtain:
  \begin{align}
    \notag &\mbox{}\,\,\,\,\,
    - \frac{1}{t} \sum_{x \in \Nz\setminus \{a\}}
    \frac{\bar{h}(x)}{\gamma_x} \frac{\lambda_x}{\lambda_a+\lambda_x} 
    \int_0^t F_x (s)  e^{-\frac{t-s}{\gamma_x}} d s\\
    \label{eq:gerador-inf-6} &=
    - \frac{1}{t} \sum_{x \not\in \Nz_\delta\setminus \{a\}}
    \frac{\bar{h}(x)}{\gamma_x} \frac{\lambda_x}{\lambda_a+\lambda_x} 
    \int_0^t F_x (s)  e^{-\frac{t-s}{\gamma_x}} d s\\
    \label{eq:gerador-inf-7} &\quad
    - \frac{1}{t} \sum_{x \in \Nz_\delta\setminus \{a\}}
    \frac{\bar{h}(x)}{\gamma_x} \frac{\lambda_x}{\lambda_a+\lambda_x} 
    \int_0^t F_x (s)  e^{-\frac{t-s}{\gamma_x}} d s
  \end{align}

  Note that $|\bar{h}(x)|/\gamma_x$ is bounded since $h \in
  \DD$. Thus we bound the absolute value of
  \eqref{eq:gerador-inf-7} by:
  \begin{displaymath}
    \text{const} \sum_{x \in \Nz_\delta\setminus\{a\}}
    \lambda_x \frac{1}{t} \int_0^t F_x (s)  d s.
  \end{displaymath}

  This quantity vanishes as $t \searrow 0$ since
  $\Gamma^{<x,a>}(S_x)$ is a continuous positive random variable and
  $\sum_{x \in \Nz_\delta}\lambda_x < \infty$.

  We can bound \eqref{eq:gerador-inf-6} above and below by $- (L
  \pm \epsilon) G(t)/t$, where
  \begin{gather*}
    G(t) := \sum_{x \not\in \Nz_\delta\setminus \{a\}}
    \frac{\lambda_x \gamma_x}{\lambda_a+\lambda_x} \int_0^t F_x (s)
    \frac{1}{\gamma_x}e^{-\frac{t-s}{\gamma_x}} d s.
  \end{gather*}

  Now we are left with showing that $G(t)/t \to 1$ as $t \searrow
  0$.

  The integral in the definition of $G$ is a probability distribution function,
  since it is a convolution of a probability distribution function and a
  probability density function. Thus $G$ is a distribution function of a
  finite measure. Our strategy now is to use Karamata's Tauberian
  Theorem (as stated in Theorem XIII.5.1 of \cite{fellerv2}), that
  relates the behavior of the $G$ near the origin with the behavior
  of its Laplace transform near $+\infty$.

  Let us denote the Laplace transform of $G$ by $\tilde{G}$. For $\beta
  > 0$, integrating by parts, we can write:
  \begin{align}
    \notag
    \tilde{G}(\beta) :=&
    \int_0^\infty e^{-\beta t} G(d t)
    = \beta \int_0^\infty e^{-\beta t} G(t) d t\\
    \label{eq:gerador-inf-8}
    =& \sum_{x \not\in \Nz_\delta\setminus \{a\}}
    \frac{\lambda_x}{\lambda_a+\lambda_x} \beta 
    \int_0^\infty e^{-t \beta}
    \int_0^t F_x (s) e^{-\frac{t-s}{\gamma_x}} d s d t.
  \end{align}

  Note that the integral in $t$ at the right hand side of
  \eqref{eq:gerador-inf-8} is the Laplace transform of the convolution
  of two functions. We can compute the Laplace transform of each one
  of them, getting:
  \begin{align*}
    \int_0^\infty e^{-\beta t} e^{-t/\gamma_x} d t &= \frac{\gamma_x}{1 + \beta \gamma_x},\\
    \int_0^\infty e^{-\beta t} F_x(t) d t = \frac{1}{\beta} \E \left[e^{- \beta \Gamma^{<x, a>} (S_x)}\right]
    &= \frac{\lambda_a+\lambda_x}{\beta} \left(\lambda_1 + \lambda_x 
    + \beta \sum_{y \in \Nz\setminus \{a, x\}}\frac{\lambda_y \gamma_y}{1+\beta \gamma_y}\right)^{-1}.
  \end{align*}

  Thus
  \begin{align*}
    \tilde{G}(\beta) &=
    \sum_{x \not\in \Nz_\delta\setminus \{a\}} 
    \frac{\lambda_x \gamma_x}{1 + \beta \gamma_x}
    \left(
      \lambda_1 + \lambda_x + \beta \sum_{y \in \Nz\setminus \{a, x\}}
      \frac{\lambda_y \gamma_y}{1+\beta \gamma_y}
    \right)^{-1}.
  \end{align*}
  
  Since $\sup_{x \in \Nz} \lambda_x < \infty$ and $\sum_{x \in
    \Nz_\delta} \lambda_x < \infty$, we get that $\beta \tilde{G}(\beta) \to
  1$ as $\beta \to \infty$. We have thus verified the assumption of
  the Tauberian Theorem:
  \begin{displaymath}
    \frac{\tilde{G}(\alpha \beta)}{\tilde{G}(\beta)}
    \xrightarrow{\beta \to \infty} \frac{1}{\alpha},
  \end{displaymath}
  yielding that $G(t)/\tilde{G}(1/t) \to 1$ as $t \searrow 0$. With
  this we can conclude:
  \begin{displaymath}
    \frac{G(t)}{t} =
    \frac{G(t)/\tilde{G}(1/t)}{\left(\frac{1}{t}\tilde{G}(\frac{1}{t})\right)^{-1}}
    \xrightarrow{t \searrow 0} 1.
    \qedhere
  \end{displaymath}
\end{proof}

\begin{corolario}
  \label{cor:gerador-inf}
  Theorem \ref{teo:gerador-inf} also holds when $\sup_{x
    \in \Nz} \lambda_x = \infty$.
\end{corolario}

\begin{proof} (Sketchy)
  The basic idea is to {\em break} sites with large $\lambda_x$ into
  several sites with bounded $\lambda_x$. To formalize this argument, 
  we will construct
  simultaneously a K process with the original parameters, and another
  K process related to the first, but with bounded weights.

  Take $\Nz^\prime = \{ (x, n) \in \Nz \times \N: n < \lambda_x\} $;
  this is a countable set. For $(x, n) \in \Nz^\prime$ let us define:
  \begin{gather*}
    \lambda^\prime_{(x, n)} := \frac{\lambda_x}{\lceil \lambda_x
      \rceil},\quad
    \gamma^\prime_{(x,n)} := \gamma_x
  \end{gather*}
  
  It can be verified that $\{\lambda^\prime_{(x, n)}, \gamma^\prime_{(x, n)}: (x, n)
  \in \Nz^\prime\}$ satisfy 
  \eqref{eq:soma-lambda-gamma}. Furthermore:
  \begin{gather*}
    \sum_{n < \lambda_x} \lambda^\prime_{(x, n)} = \lambda_x,\quad
    \sup_{(x, n)\in \Nz^\prime}\lambda^\prime_{(x,n)} \leq 1
  \end{gather*}
  
  Since the superposition of independent Poisson processes is a Poisson
  process, we may construct $\{X(t), t \geq 0\}$ and
  $\{X^\prime(t), t \geq 0\}$, K processes started at $\infty$ in the
  same probability space such that:
  \begin{itemize}
  \item $X$ has $\Nzb$ as the state space and parameters $\{\lambda_x,
    \gamma_x: x\in \Nz\}$;
  \item $X^\prime$ has $\Nzb^\prime=\Nz^\prime\cup\{\infty\}$ as the state space and parameters
    $\{\lambda^\prime_{(x,n)}, \gamma^\prime_{(x,n)}: (x,n) \in
    \Nz^\prime\}$;
  \item $X(t)$ is the projection in the first coordinate of $X^\prime(t)$.
  \end{itemize}

  Let $\DD$ and $\DD^\prime$ be the domain of the generator as defined
  in \eqref{eq:classe-D} for the processes $X$ and $X^\prime$
  respectively.

  Take $h \in \DD$; we want to show \eqref{eq:gerador-inf}. For that
  let us define $h^\prime: \Nzb^\prime \to \R$ by $h^\prime((x, n)) =
  h(x)$, $h^\prime(\infty) = h(\infty)$. With this construction we
  have that $h(X(t)) = h^\prime(X^\prime(t))$ \qc.

  It may be readily verified that $h^\prime \in  \DD^\prime$, so we
  conclude using Theorem \ref{teo:gerador-inf} for the process
  $X^\prime$:
  \begin{align*}
    &\mbox{}\quad\quad\lim_{t \searrow 0} \frac{\E \left[ h(X(t))\right] -
      h(\infty)}{t} =
    \lim_{t \searrow 0} \frac{\E \left[ h^\prime(X^\prime(t))\right] -
      h^\prime(\infty)}{t}\\& =
    \lim_{(x,n) \to \infty}
    \frac{h^\prime(\infty) - h^\prime((x,n))}{\gamma^\prime_{(x,n)}}=
    \lim_{x \to \infty}
    \frac{h(\infty) - h(x)}{\gamma_x}.
    \qedhere
  \end{align*}

\end{proof}

\begin{teorema}
  \label{teo:gerador-suficiente}
  The operator $\AAA: \DD \to \CC_0$. Given by:
  \begin{equation}
    \label{eq:gerador-final}
    \AAA f (x) = \begin{cases}
      \frac{f(\infty) - f(x)}{\gamma_x}, & \text{if } x \in \Nz;\\
      \lim_{x \to \infty} \frac{f(\infty) - f(x)}{\gamma_x}, &
      \text{if } x = \infty,
    \end{cases}
  \end{equation}
  has a closure that is the Markov generator of the semigroup of the
  K process over $\CC_0$.
\end{teorema}

\begin{proof}
  Following Theorem 2.6 from the Chapter 1 of \cite{ethier:86},
  equipping $\CC_0$ with the \emph{supremum} norm, making it a Banach space,
  we need to verify that:
  \begin{enumerate}
  \item $\DD$ is dense in $\CC_0$;
  \item $\Vert \beta f - \AAA f \Vert \geq \beta \Vert f \Vert$ for
    every $\beta > 0$ and $f \in \DD$.
  \item The range of $I - \AAA$ is $\CC_0$, where $I$ is the
    identity operator.
  \end{enumerate}

  Let us prove each item separately:
  \begin{enumerate}
  \item Given a $g \in \CC_0$ and a $\epsilon > 0$, we want to find a
    $f \in \DD$ such that $\sup_{x \in \Nzb} |f(x) - g(x)| <
    \epsilon$. Since $g \in \CC_0$, there is a $\delta > 0$ such that
    $|g(x) - g(\infty)| < \epsilon/2$ whenever $\gamma_x < \delta$.

    Fixing such $\delta$, let us define $f(\infty) = g(\infty)$ and
    $f(x) = g(x)$ whenever $\gamma_x \geq \delta$.  By
    \eqref{eq:soma-lambda-gamma} we have
    that $\sum_{x: \gamma_x \geq \delta} \lambda_x < \infty$ and
    $\sum_{x: \gamma_x < \delta} \lambda_x = \infty$, so we can choose
    finitely many $x$ not yet chosen, and define $f(x)$ in a such a way
    that $\sum_x \lambda_x (f(x) - f(\infty)) = 0$, if we set $f(x) =
    g(\infty)$ for every other $x$. One may readily check that such an
    $f$ satisfies our conditions.
  
  \item Take $f \in \DD$ and $\beta > 0$.

    Since $\sum_{x\in \Nz} (f(x)-f(\infty)) \lambda_x = 0$, then
    $\Vert f \Vert > |f(\infty)|$, so for a fixed $\epsilon > 0$ we
    can take $y \in \Nz$ such that $|f(y)| > \Vert f \Vert - \epsilon$
    and $|f(y)| > |f(\infty)|$. This implies that $f(y)$ and $-\AAA f
    (y)$ have the same sign, so we may conclude that
    \begin{displaymath}
      \Vert \beta f - \AAA f \Vert \geq
      |\beta f(y) - \AAA f (y)| \geq
      |\beta f(y)| \geq
      \beta\Vert f \Vert - \beta\epsilon.
    \end{displaymath}

  \item Given a $g\in\CC_0$, let us define:
    \begin{align*}
      L &:=
      \left[
        \sum_{x\in\Nz} \frac{\lambda_x \gamma_x}{1+\gamma_x}
      \right]^{-1}
      \left[
        \sum_{x \in \Nz} \frac{\lambda_x \gamma_x}{1+\gamma_x} (g(x)-
        g(\infty))
      \right];\\
      f(x) &:= \begin{cases}
        g(x) \frac{\gamma_x}{1+\gamma_x} + \frac{g(\infty) +
          L}{1+\gamma_x}, & \text{if } x \in \Nz;\\
        g(\infty) + L, & \text{otherwise}.
      \end{cases}
    \end{align*}
    Direct computations verify that $f \in \DD$ and $f - \AAA f = g$.\qedhere
  \end{enumerate}
\end{proof}

\vspace{.5cm}

\noindent{\bf Acknowledgements}

This paper contain results of the Master's dissertation of the second
author, supervised by the first author.  The authors enjoyed
enlightening discussions with Milton Jara and Pierre Mathieu on the
subject and results herein.  They thank NUMEC for
hospitality. They also thank an anonymous referee for suggestions 
of improvement to an earlier version of this paper.
The work of the first author is part of USP project MaCLinC.

\singlespacing
\bibliographystyle{plainnat}
\bibliography{bibliografia}

\end{document}